\documentclass[11pt,twoside]{preprint}
\usepackage[a4paper,innermargin=1.2in,outermargin=1.2in,
bottom=1.5in,marginparwidth=1in,marginparsep
=3mm]{geometry}
\usepackage{amsmath,amsthm,amssymb,enumerate}
\usepackage{mathtools}
\usepackage{hyperref}
\usepackage{mathrsfs}
\usepackage{breakurl}
\usepackage[nameinlink,capitalize]{cleveref}
\usepackage[osf]{newtxtext}
\usepackage[varqu,varl]{inconsolata}
\usepackage[cal=boondoxo]{mathalfa}
\usepackage[bigdelims,vvarbb,cmintegrals]{newtxmath}

\usepackage[dvipsnames]{xcolor}
\usepackage{mhequ}
\usepackage{comment}
\usepackage{microtype}
\usepackage{pifont}

\colorlet{darkblue}{blue!90!black}
\colorlet{darkred}{red!90!black}

\definecolor{Green}{rgb}{0.01, 0.75, 0.24}

\renewcommand{\theequation}{\thesection.\arabic{equation}}

\newtheorem{theorem}{Theorem}[section]
\newtheorem{lemma}[theorem]{Lemma}
\newtheorem{proposition}[theorem]{Proposition}

\theoremstyle{definition}
\newtheorem{assumption}[theorem]{Assumption}
\crefname{assumption}{Assumption}{Assumptions}

\newtheorem{definition}[theorem]{Definition}

\theoremstyle{remark}
\newtheorem{remark}[theorem]{Remark}

\newcommand{\vn}[1]{{\vert\kern-0.23ex\vert\kern-0.23ex\vert #1 
    \vert\kern-0.23ex\vert\kern-0.23ex\vert}}
\usepackage{mathtools}

\DeclarePairedDelimiter\floor{\lfloor}{\rfloor}

\allowdisplaybreaks

\newcommand\bone{\mathbf{1}}

\newcommand{\indep}{\perp \!\!\! \perp}

\newcommand\cB{\mathcal{B}}

\newcommand\cA{\mathcal{A}}

\newcommand\cC{\mathcal{C}}

\newcommand\cI{\mathcal{I}}

\newcommand\cM{\mathcal{M}}
\newcommand\cmm{\cM}

\newcommand\cP{\mathcal{P}}
\newcommand\cL{\mathcal{L}}

\newcommand\cF{\mathcal{F}}
\newcommand\cff{\cF}

\newcommand{\D}{\partial}

\def\eps{\varepsilon}
\def\les{\lesssim}

\newcommand{\R}{{\mathbb{R}}}

\newcommand{\Rd}{{\R^d}}

\newcommand{\E}{\mathbf{E}}
\newcommand{\bP}{\mathbf{P}}

\newcommand{\PP}{\bP}
\newcommand{\N}{\mathbb{N}}

\def\d{\partial}

\def\({\left(}
\def\){\right)}

\newcommand{\X}{\mathcal{X}}

\begin{document}
\title{A central limit theorem for the Euler method for SDEs with irregular drifts}

\author{Konstantinos Dareiotis$^1$, M\'at\'e Gerencs\'er$^2$ and Khoa L\^e$^3$}
\institute{University of Leeds, UK, \email{k.dareiotis@leeds.ac.uk}
\and TU Wien, Austria, \email{mate.gerencser@tuwien.ac.at} 
\and University of Leeds, UK, \email{k.le@leeds.ac.uk}}
\maketitle
\begin{abstract}
The goal of this article is to establish a central limit theorem for the Euler--Maruyama scheme approximating multidimensional SDEs with elliptic Brownian diffusion, under very mild regularity requirements on the drift coefficients. 
When the drift is H\"older continuous, we show that the limiting law of the rescaled fluctuations around the true solution is characterised  as the unique solution of a hybrid Young--It\^o differential equation. When the drift has positive Sobolev regularity, this limit is characterised by the solution of a transformed SDE.     
Our result is an extension of the results of  Jacod--Kurtz--Protter (1991, 1998)  in which SDEs with differentiable coefficients were considered. 
To compensate for the lack of regularity of the drifts, we utilize the regularisation effect from the non-degenerate noise.
\end{abstract}
\section{Introduction} 
\label{sec:introduction}	

We consider the asymptotic error distributions for the Euler--Maruyama scheme for stochastic differential equations of the form 
	\begin{equs}\label{sde}
		X_t=x_0+\int_0^tb(X_r)dr+\int_0^t \sigma(X_r)dB_r, \quad t\in[0,1].
	\end{equs}
	Herein, $x_0\in \Rd$, $d\ge1$, $B$ is a standard Brownian motion in $\R^d$ on a filtered probability space $(\Omega,\cff,\{\cff_t\}_{t \in [0, 1]},\PP)$, $b:\Rd\to\Rd$  and $\sigma:\Rd\to\R^{d\times d}$
	are measurable functions.  
	The Euler--Maruyama scheme for \eqref{sde} is defined by
	\begin{equs}\label{EM}
		X^n_t=x_0+\int_0^tb(X^n_{\kappa_n(r)})dr+\int_0^t\sigma(X^n_{\kappa_n(r)})dB_r, \quad t\in[0,1],
	\end{equs}
	with the notation $\kappa_n(r)=\floor{nr}/n$.
	When $b,\sigma$ are bounded with bounded first order  derivatives, it is shown  in  \cite{MR1119837,Jacod_Protter} that $X^n$ satisfies a central limit theorem. Namely, the probability law on $C([0, 1]; \R^d)$ of the process
	\begin{equs}\label{def.Vn}
		V^n_t=\sqrt n(X_t-X^n_t)
	\end{equs}
	 converges to the law of the solution $V=(V^{(1)},\ldots, V^{(d)})$ to the system of equations
	\begin{equation} \label{eq:V}
		\begin{aligned}
			V^{(\ell)}_t&=\int_0^t \partial_i b^{(\ell)} (X_r) V^{(i)}_r\,dr+\int_0^t \partial_j \sigma^{(\ell, i)}(X_r)V^{(j)}_r\,dB^{(i)}_r
			\\&\quad+\frac1{\sqrt2}\int_0^t(\partial_j \sigma^{(\ell, i)} \sigma^{(j, k)})(X_r)\,dW^{(k, i)}_r, \quad \ell=1,\ldots,d, 
		\end{aligned}
	\end{equation}
where we have adopted the Einstein summation convention over repeated indices.
In \eqref{eq:V}, $W$ is a new $d\times d$-dimensional Brownian motion defined on an extension of $(\Omega,\cff,\{\cff_t\}_{t \in [0, 1]},\PP)$ and is independent from the $\sigma$-field $\cff$. 
For each $(k, i)$, the process $W^{(k, i)}$ arises  as the limit in law of the process
\begin{equs}
	t\mapsto \sqrt{2n}\int_0^t (B^{(k)}_r-B^{(k)}_{\kappa_n(r)})dB^{(i)}_r.
\end{equs}


Our goal is to derive an analogous central limit theorem for equations whose drift is far from differentiable.  Namely, we deal with the case that $b$ posses regularity of order $\alpha \in (0, 1)$, measured in either a H\"older or a Sobolev scale.  In this setting, the limiting asymptotic distribution can not be directly characterized by \eqref{eq:V} because it involves the derivative of $b$, which only exists in the distributional sense, and evaluating the distribution $\d_i b^{(\ell)}$ at a given point is not a well-defined operation. 
We provide two possible ways of circumventing this problem.

When the drift possesses H\"older  regularity of order $\alpha \in (0, 1)$,   we make sense of the first integral in \eqref{eq:V} as a Young integral (\cite{young}).  
This relies on the observation, which is a reoccurring one in the study of averaging along oscillatory processes, that even though $\d_i b^{(\ell)}(X_r)$ is not well-defined, its integral with respect to $r$ can be given a robust meaning. Indeed, it turns out that the process 
\begin{equs}
 L^{(i,\ell)}_t=\int_0^t \partial_i b^{(\ell)}(X_r)\,dr 
\end{equs}
can be defined and it is $\beta$-H\"older continuous,  for any $\beta < (1+\alpha)/2$.  Hence,  if $V$  is $\gamma$-H\"older continuous, for some  $\gamma > 1 -  (1+\alpha)/2$,   then  integral of $V$ against $L$ can be defined as a  Young integral.  
Then formally one has
\begin{equs}\label{id.bdl}
\int_0^t \partial_i b^{(\ell)} (X_r) V^{(i)}_r\,dr=\int_0^tV^{(i)}\,dL^{(i,\ell)}_r.
\end{equs}
Hence, 
the system \eqref{eq:V} becomes a system of hybrid Young--It\^o differential equations. 

When $b$ only possesses regularity in a Sobolev scale,  we give an alternative approach. 
We rely on a Zvonkin type  transformation that removes the distributional term.
We show that with some function $v:[0,1]\times\R^d\to\R^{d\times d}$ such that for all $(t,x)\in[0,1]\times\R^d$, $v_t(x)$ is an invertible matrix, the process $v_t(X_t)V_t$ can be characterised as the solution of an It\^o equation.

\subsection{Notation}
Let us introduce some notation that will be frequently used throughout the paper.  
For any function $f\colon Q\to V$, where $Q\subset \R^k$ is a Borel set and $(V,|\cdot|)$ is a  normed space, with the notation $\N_0:= \N \cup\{0\}$,  let us introduce the (semi-)norms
\begin{equs}
\|f\|_{C^0(Q,V)}&=\sup_{x\in Q}|f(x)|\,;& &\\
\,[f]_{C^\gamma(Q,V)}&=\sum_{\substack{\ell\in\N_0^k\\|\ell|_1=\hat\gamma}}\sup_{x\neq y\in Q}\frac{|\d^{\ell}f(x)-\d^{\ell}f(y)|}{|x-y|^{\bar\gamma}}\,,&\qquad&\gamma>0,\,\gamma=\hat\gamma+\bar\gamma,\,\hat\gamma\in\N_0,\bar\gamma\in(0,1];\\
\|f\|_{C^\gamma(Q,V)}&=\sum_{\substack{\ell\in\N_0^k\\|\ell|_1< \gamma}} \|\D^\ell f\|_{C^0(Q,V)}+[f]_{C^\gamma(Q,V)},&\qquad &\gamma>0.
\end{equs}
In the above $|\ell|_1=\ell_1+\ldots+\ell_k$ if $\ell=(\ell_1,\ldots,\ell_k)\in \N_0^k$. 
By $C^\gamma(Q,V)$ we denote the space of all measurable functions $f : Q \to V$ such that $\|f\|_{C^\gamma(Q,V)}<\infty$. Frequently the target space will be dropped from the notation, if this does not cause confusion. For example, we will write $L_p(\Omega)$ in place of $L_p(\Omega; V)$, and $C^\gamma(\R^d)$ in place of $C^\gamma(\R^d; V)$, if the target space $V$ is understood from the context. Notice that  $C^0$  denotes the collection of bounded functions. 
Similarly, for $k \in \N$, $C^k$ functions are bounded $(k-1)$-times continuously differentiable and the derivatives of $(k-1)$-order are Lipschitz continuous. For $\gamma \in \R_+ \setminus \N_0$,  $C^\gamma$ are  of course the usual H\"older spaces.  For $\gamma \in (0, 1)$, we will denote by $C^{\gamma+}$ the closure of $C^\infty:= \cap_{k=1}^\infty C^k$ in $C^\gamma$. 
It is easy to see that $C^{\gamma+\varepsilon}\hookrightarrow C^{\gamma+}\hookrightarrow C^{\gamma}$ for every $\gamma\in(0,1)$ and $\varepsilon>0$.
For $\gamma<0$, we denote by $C^\gamma (\R^d)$ the space of all tempered distributions $f $ such that 
\begin{equs}
\|f\|_{C^\gamma(\R^d)} := \sup_{t \in (0,1)} t^{-\gamma/2} \| \cP_t f\|_{C^0(\R^d)} < \infty, 
\end{equs}
where $(\mathcal{P}_t) _{t \geq 0}$ is the heat semigroup associated to the standard Gaussian kernel 
$$p_t(x)=(2 \pi t)^{-d/2}e^{-|x|^2/(2t)}.$$ 
For a vector/ matrix $y$, its transpose is denoted by $y^*$.

Finally, a notational convention: in proofs of statements we
use the shorthand $f\lesssim g$ to mean that there exists a constant $N$ such that $f\leq N g$, and that $N$ does
not depend on any other parameters than the ones specified in the statement.

\section{Formulation and main results for H\"older drift}

We begin with the presentation of our results in the case that $b$ possesses H\"older regularity. The proofs are postponed to \cref{sec:proof_Holder}.

\begin{assumption}\label{asn:sigma}
The coefficient $\sigma$ belongs to $C^2(\R^d; \R^{d \times d})$ and  there exists $\lambda>0$ such that  $|y^* (\sigma \sigma^*)(x) y|\geq \lambda^2|y|^2$ for all $x,y\in\R^d$.
\end{assumption}

\begin{assumption}\label{asn:Holder}
For some $\alpha \in (0, 1)$, $b \in C^{\alpha+}(\R^d; \R^d)$. 
\end{assumption}

We will mostly be dealing with solutions in a weak sense. To be more precise, by a solution of \eqref{sde} we mean the following:
	
\begin{definition}
A solution of \eqref{sde} is a collection $\{ (\Omega, \cF, \mathbb{F}, \bP), X, B\}$ such that 
\begin{enumerate}[(i)]
\item $(\Omega, \cF, \bP)$ is a probability space with a filtration $\mathbb{F}=\{ \cF_t\}_{t \in [0, 1]}$.

\item  $X$ is an $\R^d$-valued process on $\Omega$, adapted to $\mathbb{F}$.

\item $B$ is an $\R^d$-valued $\mathbb{F}$-Wiener process on $\Omega$.

\item Equality \eqref{sde} is satisfied $\bP$-almost surely, for all $t \in [0, 1]$.
\end{enumerate}
\end{definition}	
It is well-known that under the above assumptions, equation \eqref{sde} admits a unique solution (see, \cite{Veret80}).  Next, we want to give a meaning to equation \eqref{eq:V} and its solution. To do this, we need the following proposition.

\begin{proposition}   \label{prop:integral_operator}
Let Assumptions \ref{asn:sigma} and \ref{asn:Holder} hold and  let $\{ (\Omega, \cF,  \mathbb{F}, \bP), X, B\}$ be a solution of \eqref{sde}.  Then, there exists a bounded linear operator 
\begin{equs}
C^{\alpha+}(\R^d)  \ni f \mapsto (Q^Xf)_{t \in [0, 1]}  \in  C^{(1+\alpha)/2}([0, 1]; L_2(\Omega; \R^d ))
\end{equs} such that
\begin{enumerate}[(i)]
\item For all $f \in C^{\alpha +}(\R^d)$, the process $(Q^Xf)_{t \in [0, 1]}$ is $\mathbb{F}$-adapted. 
\label{item:op-on-smooth}\item  For all $f \in C^1(\R^d) $,  with probability one, 
\begin{equs}
(Q^Xf)_t = \int_0^t \nabla f(X_s) \, ds, \quad \text{for all } t \in [0, 1]. 
\end{equs}
 \item \label{it:continuity_in_Lp}For all $p \geq 1$, there exists a constant $N=N(p, \alpha, d)$ such that 
\begin{equs}
\| Q^Xf\|_{C^{(1+\alpha)/2}([0, 1]; L_p(\Omega ; \R^d))} \leq N \| f\|_{C^{\alpha}(\R^d)}.
\end{equs}
\end{enumerate}
\end{proposition}

\begin{remark}    \label{rem:Holder-reg-a.s.}
Let  $f \in C^{\alpha+}(\R^d)$. By \eqref{it:continuity_in_Lp} of the above result and \cref{lem:interchange-LpCa}, it follows that for all $p \geq 1$ and $\alpha'  \in (0, (1+\alpha)/2)$,  there exists $N=N(p,\alpha,\alpha', d)$ such that
\begin{equs}
\| Q^Xf \|_{L_p(\Omega; C^{\alpha'}([0,1];  \R^d))}
 \leq N \| f\|_{C^\alpha(\R^d)}.
\end{equs}
In particular, if $Z \in C^\beta([0, 1]; \R^d)$, then integrals of the form $\int Z_r d  (Q^Xf)_r$  can be defined as Young integrals provided that $\beta+\alpha>1$. 
\end{remark}

\begin{remark}               \label{rem:In_C_alpha_almost_surely}
If $f \in C^{\alpha+}(\R^d; \R^d)$ we will denoted by $Q^Xf$ the $\R^{d\times d}$-valued process given by $(Q^Xf)^{(i, j)}:=(Q^Xf^{(i)})^{(j)}$,  $i, j \in  \{1, ..., d\}$. 
\end{remark}

\begin{definition}  \label{def:solution_of_system}
A solution of the system  \eqref{sde},\eqref{eq:V} is a collection $\{ (\Omega, \cF,  \mathbb{F}, \bP),  X, V, B, W\}$ such that 
\begin{enumerate}[(i)]
\item $\{(\Omega, \cF, \mathbb{F}, \bP), X,  B\}$ is a solution of \eqref{sde}. 

\label{item:def_sol_1}

\item  $W$ is an $\R^{d\times d}$-valued $\mathbb{F}$-Wiener process on $\Omega$,  independent of $(X,B)$. 
\label{item:def_sol_2}

\item $V$ is an $\mathbb{F}$-adapted process and there exists $\beta > (1-\alpha)/2$ such that 
$\bP \big( V \in C^\beta([0, 1]; \R^d)\big) =1$. 
\label{item:def_sol_3}

\item   \label{item:def_sol_4} For all  $\ell \in \{1, ..., d\}$,   
\begin{equs}   \label{eq:Equation_V_Holder_indices}
			V^{(\ell)}_t&=
			\int_0^t V^{(j)}_r\,d(Q^X b )^{( \ell ,j)}_ r+
			\int_0^t \partial_j \sigma^{(\ell, i)}(X_r)V^{(j)}_r\,dB^{(i)}_r
\\
&\quad+
\frac1{\sqrt2}\int_0^t(\sigma^{(j, k)} \partial_j \sigma^{(\ell, i)} )(X_r)\,dW^{(k, i)}_r,
	\end{equs}
 $\bP$-almost surely, for all $t \in [0, 1]$. 
\end{enumerate}
\end{definition}

To ease the notation, we will use the shorthand version of \eqref{eq:Equation_V_Holder_indices} 
\begin{equs}
    V_t= \int_0^t V_r \,d(Q^Xb )_r+\int_0^t \nabla\sigma  (X_r)V_r\,dB_r
+\frac1{\sqrt2}\int_0^t\sigma  \nabla \sigma (X_r)\,dW_r,  
\end{equs}
where the multiplication of matrices/tensors is understood by \eqref{eq:Equation_V_Holder_indices}. 
	
For the next theorem let us introduce the notation 
\begin{equs}
    \mathcal{C}:= \big(C([0,1]; \R^d)\big)^3 \times  C([0,1]; \R^{d\times d}).
\end{equs}
          
\begin{theorem}        \label{thm:existence_uniqueness}   
There exists a solution  $\{ (\Omega, \cF,  \mathbb{F}, \bP),  X, V, B, W\}$ of the system \eqref{sde},\eqref{eq:V}.  Moreover, if $\{ (\Omega', \cF',  \mathbb{F}', \bP'),  X', V', B', W'\}$ is another  solution, then the laws of $(X, V, B, W)$ and  $(X', V', B', W')$ on $\mathcal{C}$ coincide.  
\end{theorem}

The law of the unique solution $(X, V, B, W)$  of the system \eqref{sde},\eqref{eq:V} will be denoted by $\mu$. 
We are now in position to state our main result for H\"older drifts. 

\begin{theorem} \label{thm:main_Holder}
Let  Assumptions \ref{asn:sigma} and \ref{asn:Holder} and let $\{ (\Omega, \cF, \mathbb{F}, \bP), X, B\}$ be a solution of \eqref{sde}. Define the processes $X^n, W^n$, and $V^n$ by 
\begin{equs}
dX^n_t & = b(X^n_{\kappa_n(t)})dt+ \sigma(X^n_{\kappa_n(t)})dB_t,  \qquad X^n_0=x_0, 
\\
dW^n_t & =   \sqrt{2n}  (B_t-B_{\kappa_n(t)})  \otimes \, dB_t, \qquad W^n_0=0, 
\\
V^n_t&= \sqrt{n}(X_t-X^n_t).
\end{equs}
Then, the law of $(X, V^n,  B, W^n)$ on $\mathcal{C}$ converges weakly to $\mu$. 
\end{theorem}

\begin{remark}
The endpoint case $\alpha=0$, corresponding to merely bounded and measurable $b$, appears to be a quite challenging problem, in fact it is not even clear whether the factor $\sqrt{n}$ is the correct one.
Even a much simpler problem, determining the sharp rate of convergence of
\begin{equs}
I_n:=\int_0^1 \big(b(B_t)-b(B_{\kappa_n(t)})\big) \,dt
\end{equs}
for general bounded measurable function $b$ is an open question. 
The best known rate is the bound $\|I_n\|_{L^p(\Omega)}\lesssim n^{-1/2}\sqrt{\log n}$ for all $p\ge2$. 
The case $p=2$ is established in \cite{DG} by direct estimations while the case $p>2$ is established in \cite{leBMO} using the known estimate for $p=2$ and the John--Nirenberg inequality for stochastic processes of bounded mean oscillation. 
\end{remark}

\section{Formulation and main results for Sobolev  drift}
	
In this section, we present our results in the case that $b$ possesses positive regularity in a Sobolev scale. More precisely, \cref{asn:Holder} is replaced by the following \cref{asn:Sobolev}, while \cref{asn:sigma}  remains in force. The proofs are postponed to \cref{sec:proofs_Sobolev}.
\begin{assumption}\label{asn:Sobolev}
For some $\alpha>0$ and $m\geq 1$ the coefficient $b$  belongs to $W^\alpha_m(\R^d; \R^d)$, that is,   
\begin{equ}
 \| b\|_{W^\alpha_m(\R^d; \R^d)}^m:=\int_{\R^d} |b(x)|^m\, dx + \int_{\R^d}\int_{\R^d}\frac{|b(x)-b(y)|^m}{|x-y|^{d+\alpha m}}\,dx\,dy<\infty.
\end{equ}
Moreover, $b$ is bounded. 
\end{assumption}

\begin{remark}
    Notice that if Assumption \ref{asn:Sobolev} is satisfied, then for any $m' \geq 1$ there exists $\alpha' \in (0, \alpha)$ such that $b \in W^{\alpha'}_{m'}(\R^d; \R^d)$. Therefore, without loss of generality we can assume that $m \geq \max\{2, d\}$. 
\end{remark}

The limiting distribution in this case will be characterised by means of the Zvonkin transform. In order to proceed with the statement, we will  need some standard facts from PDE theory.  Under \cref{asn:Sobolev}, $b$ belongs to $L_{p_0}(\Rd; \R^d)$ for all $p_0\ge m$. 
  Given  $\ell \in \{1, \ldots, d\}$ and $\theta>0$, we consider the equation
	\begin{equs}\label{eqn.pde}
	\partial_t u^{(\ell)}+\frac{1}{2}a^{(i,j)}\partial^2_{i,j} u^{(\ell)}+b\cdot\nabla u^{(\ell)}= \theta u^{(\ell)}-b^{(\ell)},\quad u^{(\ell)}(1,\cdot)=0.
	\end{equs}
 Under Assumptions \ref{asn:sigma} and \ref{asn:Sobolev},
 by \cite[Theorem 10, p.123]{Krylov_PDE} there exists a unique  $v^{(\ell)}\in W^{1,2}_{p_0}((0, 1) \times \R^d)$, where the latter denotes the space of all measurable,  locally integrable functions  $v :[0, 1] \times \R^d \to \R$ such that 
 \begin{equs}
     \| v\|^{p_0} _{W^{1,2}_{p_0}((0, 1) \times \R^d)}:= \int_0^1 \int_{\R^d} | v|^{p_0}+|\nabla v|^{p_0}+| \nabla^2 v |^{p_0}+ | \D_tv|^{p_0} \, dx dt < \infty.
 \end{equs}

 Moreover, by  \cite[Theorem 1, p.120]{Krylov_PDE}  there exist positive  constants $N$, $\theta_0$ depending only in $d, p_0, \lambda, \| \sigma\|_{C^2}$ and $
 \|b\|_{C^0}$ such that, the estimate 
 \begin{equs}            \label{eq:W^2_p}
     \| u^{(\ell)}\|_{W^{1,2}_{p_0}((0, 1) \times \R^d)} \leq N \| b^{(\ell)} \|_{L_{p_0}((0, 1) \times \R^d)}
     \end{equs}
     holds. Notice that the solution does not depend on the choice of $p_0$ and the above estimates are valid for all $p_0 \geq m$ for the same solution. Moreover,  for any $\theta$ sufficiently large, we have 
     \begin{equs}
  \|\nabla u^{(\ell)}\|_{C^0([0, 1] \times \R^d)} \leq N \theta^{-1/2}, \label{tmp.2709regu}
 \end{equs}
with $N=N(\lambda, d, \|\sigma\|_{C^2}, \|b\|_{C^0}, \|b\|_{L_m})$ (see, for example \cite[Lemma 2.10]{DGL}).  It is also known (see, \cite[Lemma 10.2]{MR2117951}) that for each $\gamma \in (0, 1)$, if $p_0$ is large enough, then
 \begin{equs}                       \label{eq:Holder_grad_u}
     \|\nabla u ^{(\ell)} \|_{C^{\gamma/2, \gamma}([0, 1]\times \R^d)} \leq N \| b^{(\ell)} \|_{L_{p_0}(\R^d)}, 
 \end{equs}
 with $N= N(d, p_0, \lambda, \|\sigma\|_{C^2}, \| b\|_{C^0})$. 
 Finally, since in addition $b \in W^\alpha_m(\R^d ; \R^d )$, it also follows (see, \cref{lem:interpolation_PDE}) that  
\begin{equs}           \label{eq:u_in_LqLp}
    \| u^{(\ell)}\|_{L_{p_0}((0, 1); W^{2+\alpha}_m (\R^d))} \leq N \| b^{(\ell)}\|_{W^\alpha_m(\R^d)}, 
\end{equs}
with $N=N(\alpha, \lambda, d, \|\sigma\|_{C^2}, \|b\|_{C^0}, \|b\|_{L_m}, \theta)$, where
\begin{equs}
    \| u^{(\ell)}\|_{L_{p_0}((0, 1); W^{2+\alpha}_m (\R^d))}:= \| u^{(\ell)}\|_{L_{p_0}((0, 1); W^{2}_m (\R^d))}+ \| \nabla^2 u^{(\ell)}\|_{L_{p_0}((0, 1); W^\alpha_m (\R^d))}. 
\end{equs}
Further, let us set $u=(u^{(1)}, ..., u^{(d)})$ and let us fix a $\theta$
sufficiently large, so that the matrix $I- \nabla u (t, x)$ is invertible and both $I- \nabla u (t, x)$ and $(I- \nabla u (t, x))^{-1}$ are bounded uniformly in $(t, x)$, which can be done because of \eqref{tmp.2709regu}. 

With this setting, we consider \eqref{sde}  coupled with 
\begin{equs}
    V^{(\ell)}_t-
    \D_{j} u^{(\ell)}(t, X_t) V^{(j)}_t  &= \theta  \int_0^t \, \D_j u^{(\ell)}(r, X_r) V^{(j)}_r \, dr 
\\
  &\quad+ 
  \int_0^t  \big( \D_j  (\D_\rho u^{(\ell)} \sigma^{(\rho, i)}) (r, X_r) V^{(j)}_r + \D_j \sigma^{(\ell, i)}(X_r) V^{(j)}_r \big)  \, dB^{(i)}_r 
    \\
    &\quad+ \frac{1}{\sqrt{2}} 
    \big( ( I^{(\ell, \rho)}+ \D_\rho u^{(\ell)} ) \sigma^{(j, k)}\D_j \sigma^{(\rho, i)} \big)(t, X_t) \, dW^{(k, i)} 
    \label{eq:V_Sbolev_indices}
\end{equs}
or, in matrix notation, 
	\begin{equs}    \label{eq:V_Sbolev_simplified}
		(I-\nabla u(t,X_t))V_t&=  \theta\int_0^t\nabla u(r,X_r)V_rdr
		+\int_0^t[\nabla \big( (\nabla u + I ) \sigma \big) ](r,X_r)V_rdB_r
		\nonumber\\&\quad+\frac1{\sqrt2}\int_0^t(I+\nabla u)\sigma \nabla \sigma  (r,X_r) dW_r,  
	\end{equs}
 where the precise meaning of multiplication of tensors/matrices is understood by \eqref{eq:V_Sbolev_indices}. 

\begin{definition}                 \label{def:solution_of_system_Sobolev}
A solution of the system  \eqref{sde},\eqref{eq:V_Sbolev_indices} is a collection $\{ (\Omega, \cF,  \mathbb{F}, \bP),  X, V, B, W\}$ such that 
\begin{enumerate}[(i)]
\item $\{(\Omega, \cF, \mathbb{F}, \bP), X,  B\}$ is a solution of \eqref{sde}. \label{item:def_sol_1_Sobolev}

\item  $W$ is an $\R^{d\times d}$-valued $\mathbb{F}$-Wiener process on $\Omega$,  independent of $(X,B)$. 
\label{item:def_sol_2_Sobolev}

\item $V$ is a continuous  $\mathbb{F}$-adapted process.
\label{item:def_sol_3_Sobolev}

 \item  \label{item:def_sol_4_Sobolev} Equation \eqref{eq:V_Sbolev_indices} is satisfied $\bP$-almost surely, for all $t \in [0, 1]$. 
\end{enumerate}
\end{definition}

\begin{theorem}        \label{thm:existence_uniqueness_Sobolev}   
There exists a solution  $\{ (\Omega, \cF,  \mathbb{F}, \bP),  X, V, B, W\}$ of the system \eqref{sde},\eqref{eq:V_Sbolev_indices} .  Moreover, if $\{ (\Omega', \cF',  \mathbb{F}', \bP'),  X', V', B', W'\}$ is another  solution, then the laws of $(X, V, B, W)$ and  $(X', V', B', W')$ on $\mathcal{C}$ coincide. 
\end{theorem}
 The law of the unique solution $(X, V, B, W)$ of the system \eqref{sde},\eqref{eq:V_Sbolev_indices} will be denoted by $\tilde{\mu}$. 
\begin{theorem}
\label{thm:main_theorem_Sobolev}
Let Assumptions \ref{asn:sigma} and \ref{asn:Sobolev} hold. With the notation of \cref{thm:main_Holder} we have that the laws $(X, V^n, B, W^n)$ converge weakly to $\tilde{\mu}$.
    
\end{theorem}

\section{Quadrature Estimates}\label{sec.quadrature}

In this section we prove some quadrature-type estimates which will be used for the proofs of our main results. More precisely, for  functions $g=g(t,x)$ and $f=f(x)$, and $(s,t)\in[0,1]_{\leq}^2:= \{ (s, t) \in [0, 1]^2; s \leq t \}$, we estimate moments of the following integral
\begin{align*}
    \int_s^t g (r, X^n_r) & \big(  f(X^n_r)-f(X^n_{\kappa_n(r)})\big)\, dr.
\end{align*}
In the following two results, when one of $g,f$ has positive Sobolev regularity and the other one has positive H\"older regularity, 
we show that  moments of all positive orders of the above integral is of order $n^{-(1+\eps)/2}|t-s|^{1/2}$ for some $\eps>0$.

\begin{lemma}\label{lem:quad1}
Let $b \in C^0(\R^d; \R^d)$, $\sigma$ satisfy \cref{asn:sigma},  and let  $p\geq 2$,  $\theta \in (0, 1)$, $\rho \in [1, \infty]$. Suppose that $X$, $X^n$ satisfy \eqref{sde} and \eqref{EM}, respectively.  Then, there exist $q_0 \geq 1$, $\beta_0 \in (0, 1)$, and $\eps>0$ such that  for all $q \geq q_0, \beta  \geq \beta_0$,  $g \in L_q( (0, 1) ; C^\beta (\R^d))$,  $f \in C^0(\R^d)  \cap W^\theta_\rho(\R^d)$,   $(s,t ) \in [0,1]_{\leq}^2$,
and   $n\in\N$,  one has the bound 
\begin{equs}
 \big\|\int_s^t g (r, X^n_r) & \big(  f(X^n_r)-f(X^n_{\kappa_n(r)})\big)\, dr\big\|_{L_p(\Omega)}
 \\
 & \leq   N \|g\|_{L^q((0, 1);C^\beta(\R^d))}  \big(\|f\|_{W^{\theta}_\rho(\R^d)}+\|f\|_{C^0(\R^d)}\big) n^{-(1+\eps)/2}|t-s|^{1/2},
\label{eq:sup-quadr-add}
\end{equs} 
where 
 $N$ is a constant depending only on $\theta, \rho, p, \|b\|_{C^0}, \|\sigma\|_{C^2}, d$, and $\lambda$. 
\end{lemma} 

\begin{proof}
Let us denote by $\bar{X}^n$ the Euler scheme corresponding to $b\equiv 0$. 
By Girsanov's transform, it suffices (see \cite[Corollary 3.2]{DGL}) to prove the claim for    $\bar{X}^n$ . Estimate  \eqref{eq:sup-quadr-add} is obtained by interpolation between the cases $ f \in C^0(\R^d)  \cap L_\rho(\R^d)$ and $f \in C^0(\R^d)  \cap W^k_\rho(\R^d)$ for some large integer $k$. The estimate corresponding to each class is obtained through the stochastic sewing lemma.

Firstly, let us fix a $k  \in \mathbb{N} $ such that  $W^k_\rho(\R^d)  \subset C^2(\R^d)$ and let us chose  a $\gamma  \in (0, 1/2)$ such that $\gamma< \theta/(2k)$. 

We set 
\begin{equs}
A_{s, t} = \E^s  \int_s^tg(r,\bar{X}^n_s)   \big(  f(\bar{X}^n_r)-f(\bar X^n_{\kappa_n(r)})\big)\, dr,
\\
\mathcal{A}_t = \int_0^t g(r,\bar{X}^n_r)   \big(  f(\bar{X}^n_r)-f(\bar X^n_{\kappa_n(r)})\big)\, dr. 
\end{equs}
It has been shown in \cite{DGL} (see the proof of Lemma 3.1 therein) that for all $r>s$ we have 
\begin{equs}     \label{eq:est-pre-paper}
 |\E^s\big(  f(\bar{X}^n_r)-f(\bar{X}^n_{\kappa_n(r)})\big)| \lesssim n^{-1/2+\gamma}\| f\|_{C^0(\R^d)} |r-s|^{-1/2+\gamma}.
\end{equs}
Using  the regularity of $g$ we see that 
\begin{equs}
\|A_{s, t}\|_{L_p(\Omega) } & \lesssim n^{-1/2+\gamma}\| f\|_{C^0(\R^d)} \int_s^t \|g(r, \cdot)\|_{C^\beta(\R^d)}    |r-s|^{-1/2+\gamma} \, dr. 
\\
& \lesssim n^{-1/2+\gamma} \|g\|_{L^q((0, 1);C^\beta(\R^d))} \|f\|_{C^0(\R^d)}  |t-s|^{(1+\gamma)/2}, 
\end{equs}
provided that  $q$  is sufficiently large.  Moreover, we have 
\begin{equs}
\E^s \delta A_{s, u, t} =  \E^s  \int_u^t ( g(r,\bar{X}^n_s)-g(r,\bar{X}^n_u) )   \E^u  \big(  f(\bar{X}^n_r)-f(\bar X^n_{\kappa_n(r)})\big)\, dr.
\end{equs}
Using again \eqref{eq:est-pre-paper} combined with the fact that $g \in L_q( (0, 1); C^\beta(\R^d))$ and that $\bar X^n_s -\bar X^n_u $ is of order $| s-u|^{1/2}$ in $L_p(\Omega)$ uniformly in $n$, it gives 
\begin{equs}
\| \E^s \delta A_{s, u, t}\|_{L_p(\Omega)}  & \lesssim \| f\|_{C^0}    n^{-1/2+\gamma} |u-s|^{\beta/2} \int_u^t  \|g(r, \cdot)\|_{C^\beta} |r-u|^{-1/2+\gamma} \, dr  
\\
& \lesssim \| f\|_{C^0(\R^d)}  \|g\|_{L^q((0, 1);C^\beta(\R^d))} n^{-1/2+\gamma} |u-s|^{\beta/2} |t-u|^{-1/2+\gamma+1/q^*}
\\
& \lesssim\| f\|_{C^0(\R^d)}  \|g\|_{L^q((0, 1);C^\beta(\R^d))} n^{-1/2+\gamma}  |t-s|^{-1/2+\gamma+1/q^*+\beta/2},    \label{eq:delta-1}
\end{equs}
where $q^*=q/(q-1)$.  Now we fix $\beta_0 \in (0, 1) $ sufficiently close to $1$ and $q_0$ sufficiently large so that
\begin{equs}    \label{eq:choice_of_gamma}
-1/2+\gamma+1/q_0^*+{\beta_0}/2 >1, 
\end{equs} 
and notice that the same relation holds for all $\beta \geq \beta_0$ and $q \geq q_0$. 
Moreover, by using the boundedness of $f$ and the regularity of $g$, it is easy to see that 
\begin{equs}
\| \mathcal{A}_t &-\mathcal{A}_s- A_{s, t}\|_{L_p(\Omega)}
\\
 & =  \big\| \int_s^tg(r,\bar{X}^n_r)   \big(  f(\bar{X}^n_r)-f(\bar X^n_{\kappa_n(r)})\big)\, dr - \E^s  \int_s^tg(r,\bar{X}^n_s)   \big(  f(\bar{X}^n_r)-f(\bar X^n_{\kappa_n(r)})\big)\, dr \big\|_{L_p(\Omega)}
\\
& \leq  4 \| f\|_{C^0} \int_s^t \|g(r, \cdot)\|_{C^0} \, dr 
\\
&\leq 4 \| f\|_{C^0(\R^d)} \|g\|_{L_q( (0, 1); C^\beta)} |t-s|^{1/q^*},
\end{equs}
and notice that $1/q^*>1/2$ by \eqref{eq:choice_of_gamma}.  Moreover, by the regularity of $g$, we get 
\begin{equs}
\| \E^s \big( \mathcal{A}_t &-\mathcal{A}_s- A_{s, t} \big) \|_{L_p(\Omega}
\\
    & \big\| \int_s^t\big( g(r,\bar{X}^n_r) - g(r,\bar{X}^n_s) \big) \E^s \big(  f(\bar{X}^n_r)-f(\bar X^n_{\kappa_n(r)})\big)\, dr \big\|_{L_p(\Omega)}
    \\
   &  \leq 2 \| f\|_{C^0(\R^d)} \int_s^t \| g(r, \cdot )\|_{C^\beta(\R^d)} |r-s|^{\beta/2} \, dr 
   \\
    &  \leq 2 \| f\|_{C^0(\R^d)} \|g\|_{L_q( (0, 1); C^\beta(\R^d))} |t-s|^{1/q^*+\beta/2},
\end{equs}
and notice that $1/q^*+\beta/2>1$ from \eqref{eq:choice_of_gamma}. 
Consequently, 
the stochastic sewing lemma  gives
\begin{equs}
& \big\|   \int_s^tg(r,\bar{X}^n_r)   \big(  f(\bar{X}^n_r)-f(\bar{X}^n_{\kappa_n(r)})\big)\, dr \|_{L_p(\Omega)} =  \| \mathcal{A}_t -\mathcal{A}_s\|_{L_p(\Omega)}
\\
   &\lesssim \| f\|_{C^0(\R^d)}   \|g\|_{L^q((0, 1);C^\beta(\R^d))} n^{-1/2+\gamma}  |t-s|^{1/2}
\\
 &\lesssim\|g\|_{L^q((0, 1);C^\beta(\R^d))} (  \| f\|_{C^0(\R^d)} + \| f\|_{L_\rho(\R^d)})    n^{-1/2+\gamma}  |t-s|  ^{1/2}, \label{eq:to-be-interpolated}
\end{equs}
which gives the estimate in the case $f \in C^0(\R^d) \cap L_\rho(\R^d)$. 

We continue with obtaining an estimate in the case $f \in \in C^0(\R^d) \cap W_\rho^k(\R^d)$. 
For  $|r-s| \geq 4/n$, by writing It\^o's formula for $f(\bar{X}^n_r)$  is easy to see that 
\begin{equs}
 |\E^s\big(  f(\bar{X}^n_r)-f(\bar X^n_{\kappa_n(r)})\big)| \lesssim n^{-1}  \| f\|_{C^2(\R^d)}.
\end{equs}
On the other hand,  if $|r-s| \leq 4/n$, then we have the trivial estimate 
\begin{equs}         \label{eq:better-rate}
\|  f(\bar{X}^n_r)-f(\bar X^n_{\kappa_n(r)}) \|_{L_p(\Omega)} \leq N  \|f\|_{C^1(\R^d)} n^{-1/2} \lesssim \|f\|_{C^1(\R^d)} n^{-1+\gamma} |r-s|^{-1/2+\gamma}.
\end{equs}
Consequently, as before we get 
\begin{equs}
\|A_{s, t}\|_{L_p(\Omega) } &\lesssim  \|g\|_{L^q((0, 1);C^\beta(\R^d))}   \| f\|_{C^2(\R^d)} n^{-1+\gamma} |t-s|^{(1+\gamma)/2},
\end{equs}
for $q$ sufficiently large. 
Moreover,  similarly to the estimate \eqref{eq:delta-1}, this time using \eqref{eq:better-rate}  which is true for all $p$, and H\"older's inequality, we arrive at 
\begin{equs}
\| \E^s \delta A_{s, u, t}\|_{L_p(\Omega)}  &\lesssim \|g\|_{L^q((0, 1);C^\beta(\R^d))}  \| f\|_{C^2(\R^d)}  n^{-1+\gamma} |t-s|^{-1/2+\gamma+1/q^*+\beta/2},
\end{equs}
and as before, the stochastic sewing lemma guarantees that 
\begin{equs}
&  \qquad \big\|   \int_s^tg(r,\bar{X}^n_r)   \big(  f(\bar{X}^n_r)-f(\bar{X}^n_{\kappa_n(r)})\big)\, dr \|_{L_p(\Omega)}
\\
&  \lesssim   \|g\|_{L^q((0, 1);C^\beta(\R^d))}  \| f\|_{C^2(\R^d)}  n^{-1+\gamma}  |t-s|^{1/2}
\\
& \lesssim \|g\|_{L^q((0, 1);C^\beta(\R^d))}  (  \| f\|_{C^0(\R^d)} + \| f\|_{W^k_\rho(\R^d)})    n^{-1+\gamma}  |t-s|^{1/2},
\end{equs}
which gives us an estimate for the case $f \in C^0(\R^d) \cap W^k_\rho(\R^d)$. 
The above bound and  \eqref{eq:to-be-interpolated}, by means of interpolation show that for $f \in C^0(\R^d) \cap W^\theta_\rho(\R^d)$, we have 
\begin{equs}
& \big\|   \int_s^tg(r,\bar{X}^n_r)   \big(  f(\bar{X}^n_r)-f(\bar{X}^n_{\kappa_n(r)})\big)\, dr \|_{L_p(\Omega)} 
\\
 &\lesssim  \|g\|_{L^q((0, 1);C^\beta(\R^d))}  (  \| f\|_{C^0(\R^d)} + \| f\|_{W^\theta_\rho(\R^d)})    n^{(-1+\gamma)\theta/k + (1-\theta/k)(-1/2+\gamma)}  |t-s|^{1/2}.
 \\
 &=    \|g\|_{L^q((0, 1);C^\beta(\R^d))}  (  \| f\|_{C^0(\R^d)} + \| f\|_{W^\theta_\rho(\R^d)})    n^{\gamma -1/2-\theta/(2k)}  |t-s|^{1/2}.
\end{equs}
Finally notice that since $\gamma < \theta/(2k)$, we have that $\gamma -1/2-\theta/(2k)< -1/2$, which finishes the proof. 
\end{proof}

\begin{lemma}\label{lem:(ii)}
Let $b \in C^0(\R^d; \R^d)$, $\sigma$ satisfy \cref{asn:sigma}, $p\geq 2$, $m \geq 1$  and $f \in C^1(\R^d)$. Suppose that $X$, $X^n$ satisfy \eqref{sde} and \eqref{EM}, respectively.  There exists $q_0 \geq 0 $  and $\eps>0$ such that  for all $q \geq q_0$,   $g \in L_q((0, 1); W^\alpha_m(\R^d)  \cap L_q(\R^d) )$,   $(s,t ) \in [0,1]_{\leq}^2$,   and  $n\in\N$ one has the bound 
\begin{equs}
 \big\|\int_s^t g (r, X^n_r) \big(  f(X^n_r)-f(X^n_{\kappa_n(r)})\big)\, dr\big\|_{L_p(\Omega)}\leq   N  n^{-(1+\eps)/2}|t-s|^{1/2},
\end{equs} 
where 
 $N$ is a constant depending only on   $\|g\|_{L_q((0, 1);W^\alpha_m \cap L_q)}, \|f\|_{C^1},   \|b\|_{C^0}, \|\sigma\|_{C^2}, p, d$, and $\lambda$. 
\end{lemma} 

\begin{proof}
First, since $f$ is Lipschitz,  we have the trivial estimate 
\begin{equs}
 \big\|\int_s^t g (r, X^n_r) \big(  f(X^n_r)-f(X^n_{\kappa_n(r)})\big)\, dr\big\|_{L_p(\Omega)}
& \lesssim  \int_s^t \|g (r, X^n_r)\|_{L_{2p}(\Omega)}  \|X^n_r - X^n_{\kappa_n(r)} \|_{L_{2p}(\Omega)} \, dr.
\end{equs}
Recall that the density of $X^n_r$ behaves similarly to the heat kernel  at time $r$ in Lebesgue spaces (see, e.g.  Theorem 4.2 in \cite{GyK}), which combined with the above estimate  and the fact that $\|X^n_r - X^n_{\kappa_n(r)} \|_{L_{2p}(\Omega)} \leq N n^{-1/2}$, shows that for all sufficiently large $q$ we have 
\begin{equs}
 \big\|\int_s^t g (r, X^n_r) \big(  f(X^n_r)-f(X^n_{\kappa_n(r)})\big)\, dr\big\|_{L_p(\Omega)}
& \lesssim \| g\|_{L_q((0, 1) \times \R^d)} n^{-1/2} |t-s|^{1/2}
\\
&   \lesssim \| g\|_{L_q((0, 1) ; L_q(\R^d) \cap L_m(\R^d))} n^{-1/2} |t-s|^{1/2}.  \\
\label{eq:to_be_interpolated_2}
\end{equs}
On the other hand, suppose for now that $g  \in L_q((0, 1) ; L_q(\R^d) \cap W^k_m(\R^d) )$ for $k$ large so that $W^k_m(\R^d)  \subset C^1(\R^d)$.  Then, from  \cref{lem:quad1},  there exists $\eps>0$ such that for $q$ sufficiently large we have
\begin{equs}
 \big\|\int_s^t g (r, X^n_r)  \big(  f(X^n_r)-f(X^n_{\kappa_n(r)})\big)\, dr\big\|_{L_p(\Omega)}
&   \lesssim \|g\|_{L^q((0, 1);W^k_m(\R^d) )}  n^{-(1+\eps)/2}|t-s|^{1/2} 
  \\
  & \lesssim\|g\|_{L^q((0, 1);L_q (\R^d)  \cap W^k_m (\R^d))}  n^{-(1+\eps)/2}|t-s|^{1/2} .
\end{equs}
The result now follows by interpolation by means of the above bound and \eqref{eq:to_be_interpolated_2}. 
\end{proof}

\section{Proof of main results for H\"older $b$}	\label{sec:proof_Holder}

\begin{proof}[Proof of \cref{prop:integral_operator}]
It clearly suffices to consider the operator $f \mapsto (Q^Xf)$ on $C^1(\R^d)$, 
given by 
$$
(Q^Xf)_t= \int_0^t \nabla f(X_s) \, ds, 
$$
and to show that on $C^1(\R^d)$ it satisfies the estimate \eqref{it:continuity_in_Lp}.  Moreover, by virtue of Girsanov's theorem we may assume that the drift $b$ is identically zero. To this end,  let $z \in \R^d$ and let us set 
\begin{equs}
\cA_t= \int_0^t  \big( f (X_s)-f(X_s+z) \big)  \, ds, \qquad A_{s, t}= \E^s  \int_s^t  \big( f (X_r)-f(X_r+z) \big) \, dr. 
\end{equs}
Then,  by the Markov property it follows that 
\begin{equs}
A_{s, t}=   \int_s^t  \Big( \mathcal{P}_{r-s}f(X_s)-  \mathcal{P}_{r-s}f(X_s+z)  \Big) \, dr, 
\end{equs}
where $(\mathcal{P}_t)_{t \in [0,1]}$ is the semigroup of the diffusion $X$.  The following estimate is well-known (see, e.g., \cite[Theorem 6.4.3]{Lorenzi}) 
\begin{equs}
| \mathcal{P}_tf(x)-\mathcal{P}_tf(y)| \lesssim  \|f \|_{C^\alpha(\R^d)} |x-y| t^{(\alpha-1)/2},
\end{equs}
which gives 
\begin{equs}
\| A_{s, t} \|_{ L_p(\Omega)} \lesssim |z| \|f\|_{C^\alpha(\R^d)} |t-s|^{(1+\alpha)/2}.
\end{equs}
In addition, it is straightforward that $\E^s \delta A_{s, u, t}=0$.  It is also easy to see  that the germs $(s, t) \to A_{s, t}$ reconstruct the process $t \to \cA_t$ by means of the stochastic sewing lemma.  Hence we can conclude that 
\begin{equs}
\big\| \int_s^t  \big( f (X_r)-f(X_r+z) \big)  \, dr \big\|_{L_p(\Omega)}  \lesssim|z| \|f\|_{C^\alpha} |t-s|^{(1+\alpha)/2}.
\end{equs}
Since this was for arbitrary $z \in \R^d$, estimate \eqref{it:continuity_in_Lp} follows and consequently the whole proposition is proved. 
\end{proof}

\begin{proposition}                 \label{prop:pathwise_uniqueness} 
Let Assumptions \ref{asn:sigma} and \ref{asn:Holder} hold. Suppose that  $\{ (\Omega, \cF,  \mathbb{F}, \bP),  X, V, B, W\}$  and $\{ (\Omega, \cF,  \mathbb{F}', \bP),  X', V', B, W\}$  are solutions of the system \eqref{sde},\eqref{eq:V}.  Then, we have that  $X=X'$ and $V=V'$, $\bP$-almost surely.
\end{proposition}

\begin{proof}
The fact that $X=X'$ is shown in \cite{Veret80}. Let us show that $V=V'$. Recall that by definition, $V, V' \in C^\beta([0, 1]; \R^d )$ with probability one, for some $\beta  \in  ( (1-\alpha)/2, 1/2)$.  Recall also that by Remark \ref{rem:In_C_alpha_almost_surely}, we can chose $\alpha' \in (0, (1+\alpha)/2)$ such that  $L:= Q^Xb \in C^{\alpha'}([0, 1]; \R^{d \times d})$ with probability one, and $\alpha'+\beta>1$.  
It is immediate that the difference $Y=V-V'$ satisfies the equation 
	\begin{equs}
	Y^{(\ell)}_t=  \int_0^t Y^{(j)}_r dL^{(\ell, j)}_r+  \int_0^t Z^{(\ell, j, i)}_r Y^{(j)}_r \, dB^{(i)}_r,   \qquad \ell \in \{1, \ldots, d \},
	\end{equs}
 where 
$Z^{(\ell, j, i)}_\cdot= \D_j \sigma^{(\ell, i)}(X_\cdot)$. As usual, we drop the indices below for notational convenience.  Next, let us denote by  $\mathbb{G}=\{ \mathcal{G}_t \}_{ t \in [0, 1]}$, the augmentation of the filtration $\{ \sigma( \mathcal{F}_s, \mathcal{F}'_s, s \leq t ) \} _{t \in [0, 1]}$.  Furthermore, let us chose 
 $\tilde{\beta} <\beta < \tilde{\alpha} < \alpha ' $ such that $\tilde{\alpha} + \tilde{\beta}  >1$. Since $L \in C^{\alpha'}$, $Y \in C^{\beta}$, and $\tilde{\alpha}< \alpha'$, $\tilde{\beta} < \beta$, it follows that for any $\mathbb{G}$-stopping time $\rho$,  the process $t \mapsto  [L]_{C^{\tilde{
 \alpha}
 }([\rho, \rho \vee t])} $ and $t \mapsto  [Y]_{C^{\tilde{
 \beta}
 }([\rho, \rho \vee t])} $ are H\"older continuous.  In addition, they are $\mathbb{G}$-adapted. 
	Let us take $\delta_1,\delta_2\in(0,1]$ (to be specified later) and define the stopping times $\rho_0=0$ and for $n\in\N$,
	\begin{equs}
	\rho_{n}:= \inf \{ t > \rho_{n-1} : [ L]_{C^{\tilde{\alpha}}([\rho_{n-1}, t])} >\delta_1 \text{ or } t - \rho_{n-1}> \delta_2\}.
	\end{equs}
	Let us furthermore define
	\begin{equs}
	\tau:= \inf \{ t > 0 : \| Y\|_{C^{\tilde{\beta}}([0, t])} > 1 \}\wedge 1
	\end{equs}
	and set $\tau_n=\rho_n\wedge\tau$. Note that if $\|Y\|_{C^{\tilde{\beta}}([0,\tau])}=0$, then $\tau=1$ and $Y\equiv 0$ on $[0,1]$.

	Note that for any $n\in\N_0$ one has
	\begin{align}        
	 [Y ]_{C^{\tilde{\beta}}([ \tau_n, \tau_{n+1}])} \leq  \Big [   \int_{\tau_n}^\cdot Y_r dL_r  \Big ]_{C^{\tilde{\beta}}([ \tau_n, \tau_{n+1}])}  + \Big [  \int_{\tau_n}^\cdot Z_r Y_r \, dB_r \Big ]_{C^{\tilde{\beta}}([ \tau_n, \tau_{n+1}])}.   
	\end{align}
	For the first term on the right hand side, we have with probability one 
	\begin{equs}
	 \Big [    \int_{\tau_n}^\cdot Y_r dL_r  \Big ]_{C^{\tilde{\beta}}([ \tau_n, \tau_{n+1}])}&  \leq  \Big [   \int_{\tau_n}^\cdot Y_r dL_r  \Big ] _{C^{\tilde{\beta}}([ \tau_n, \tau_{n+1}])} 
	 \\
	 & \leq  N \|  Y \|_{C^{\tilde{\beta}}([ \tau_n, \tau_{n+1}])} [L]_{C^{\tilde{\alpha}}([ \tau_n, \tau_{n+1}])}
	 \\
	 &\leq N\delta_1\big(|Y_{\tau_n}|+[Y]_{C^{\tilde{\beta}}([\tau_n,\tau_{n+1}])}\big),
	\end{equs}
	where $N$ depends only on $\tilde{\alpha}, \tilde{\beta}$ and $d$. 
	For the second term, notice that
	\begin{equs}
	 \Big [ \int_{\tau_n}^\cdot Z_r Y_r \, dB_r \Big ]_{C^\alpha([ \tau_n, \tau_{n+1}])} = \Big[ \int_{0}^\cdot \bone_{ \tau_n \leq r\leq  \tau_{n+1}} Z_r Y_r \, dB_r \Big]_{C^\alpha([0, 1])}.
	\end{equs}
	Hence, keeping in mind that $\tilde{\beta} <1/2$ and  using Lemma \ref{lem:interchange-LpCa} and the Burkholder--Davis--Gundy (BDG) inequality, we have for all sufficiently large $p$
	\begin{equs}
	 \Big [ \int_{\tau_n}^\cdot Z_r Y_r \, dB_r \Big ]_{L_p(\Omega; C^{\tilde{\beta}}([ \tau_n, \tau_{n+1}]))} 
	&\leq  N \Big [  \int_{0}^\cdot \bone_{ \tau_n \leq r\leq  \tau_{n+1}} Z_r Y_r \, dB_r \Big ] _{C^{1/2} ([0, 1]; L_p(\Omega))}
	\\
	&\leq  N \| Y \|_{L_p(\Omega; L_\infty([\tau_n, \tau_{n+1}]))}
	\\
	&\leq  N   \| Y_{\tau_n} \|_{L_p(\Omega)}+ N \delta_2^{\tilde{\beta}}   [  Y ]_{L_p(\Omega; C^{\tilde{\beta}} ([ \tau_n, \tau_{n+1}]))},
	\end{equs}
	where we have used that $Z$ is a process bounded by $\| \sigma\|_{C^1}$, that  by definition,  $\tau_{n+1}-\tau_n \leq \delta_2$,  and  $N$ depends only on ${\tilde{\beta}}, p$, and $d$. 
	
	Therefore we can conclude
	$$       
	[ Y ]_{L_p(\Omega; C^{\tilde{\beta}}([ \tau_n, \tau_{n+1}]))}\leq N \| Y_{\tau_n} \|_{L_p(\Omega)} +  N(\delta_1 +  \delta_2^{\tilde{\beta}})    [  Y ]_{L_p(\Omega; C^{\tilde{\beta}}([ \tau_n, \tau_{n+1}]))},
	$$
	where $N$ does not depend on $\delta_1,\delta_2$. Upon assuming that $\delta_1$ and $\delta_2$ are sufficiently small, we get 
	$$       
	[ Y ]_{L_p(\Omega; C^{\tilde{\beta}}([ \tau_n, \tau_{n+1}]))}\leq N \| Y_{\tau_n} \|_{L_p(\Omega)},
	$$
	and by iteration this gives $[ Y ]_{L_p(\Omega; C^{\tilde{\beta}}([ 0, \tau_{n+1}]))}=0$,  for all $n \in \mathbb{N}$.
	 For almost all $\omega \in \Omega$ we have that $\tau_n(\omega)=\tau(\omega)$ for all $n$ large enough, which implies $[ Y ]_{L_p(\Omega; C^{\tilde{\beta}}([ 0, \tau]))}=0$, finishing the proof.
\end{proof}

The next proposition follows from \cref{prop:pathwise_uniqueness} and the result of Yamada and Watanabe \cite[Proposition 1]{Yamada_Watanabe} (see also \cite[Proposition 3.20, p. 309]{Karatzas}).

\begin{proposition}                 \label{prop:weak_uniqueness} 
Let Assumptions \ref{asn:sigma} and \ref{asn:Holder} hold. Suppose that  $\{ (\Omega, \cF,  \mathbb{F}, \bP),  X, V, B, W\}$  and $\{ (\Omega', \cF',  \mathbb{F}', \bP'),  X', V', B', W'\}$  are solutions of the system \eqref{sde}, \eqref{eq:V}.  Then the laws of $(X, V, B, W)$ and  $(X', V', B', W')$ coincide.  
\end{proposition}

We continue with a compactness result.
\begin{lemma}  \label{lem:uniform-bounds}
Let  Assumptions \ref{asn:sigma} and \ref{asn:Holder} hold.  Let $\{ (\Omega, \cF, \mathbb{F}, \bP), X, B\}$ be a solution of \eqref{sde} and define $X^n$ and $V^n$ be as in Theorem \ref{thm:main_Holder}. Then,  for any $p \geq 2$ and $\gamma<1/2$, we have 
\begin{equs}   \label{eq:uniform-bounds}
\sup_{n \in \mathbb{N}}\| V^n\|_{C^{1/2}([0, 1]; L_p(\Omega))}  < \infty , \qquad 
\sup_{n \in \mathbb{N}}\| V^n\|_{L_p(\Omega; C^{\gamma}([0, 1]))} < \infty.
\end{equs}

\end{lemma}

\begin{proof}
Throughout the proof the proportionality constants in $\lesssim$ depend only  on $d, \lambda, \alpha, \| b\|_{C^\alpha}, \|\sigma\|_{C^2}, p$.

We begin by proving the first inequality in \eqref{eq:uniform-bounds}. Let us set $Y^n=X-X^n$ and fix  $S , T \in [0,1]$ with $S < T$. For $s, t \in [S, T]$ with $s<t$ we have 
\begin{equs}       
Y^n_t-Y^n_s &= \int_s^t [b(X_r)- b(X^n_r)] \, dr + \int_s^t [b(X^n_r)- b(X^n_{\kappa_n (r)})] \, dr \\
&\quad +\int_s^t [\sigma (X_r)- \sigma (X^n_r)] \, d B_r+\int_s^t [\sigma (X^n_r)- \sigma (X^n_{\kappa_n(r)})] \, d B_r.     \label{eq:error-decomp}
\end{equs}  
For the first term in \eqref{eq:error-decomp}, by \cite[Lemma 6.4]{ButDarGEr} (with $\tau=1/2$)  we have that 
\begin{equs}
\big\| \int_s^t [b(X_r)- b(X^n_r)] \, dr \big\|_{L_p(\Omega)} \lesssim  |t-s| \|  Y^n  \|_{C^{1/2}([S, T]; L_p(\Omega))} + |t-s|^{1/2}\| Y^n_S\|_{L_p(\Omega)}.
\end{equs}
For the second term in \eqref{eq:error-decomp}, we have by applying  \cref{lem:quad1}  with $\rho=\infty$, we get 
\begin{equs}
\big\|\int_s^t [b(X_r^n)-b(X_{\kappa_n(r)}^n)]\, dr\big\|_{L_p(\Omega)}
\lesssim n^{-1/2-\eps/2}|t-s|^{1/2}
\end{equs}
with some $\eps>0.$
For the third term,  by the BDG inequality and the regularity of $\sigma$, one easily sees that
\begin{equs}
\big\| \int_s^t [\sigma (X_r)- \sigma (X^n_r)] \, d B_r 
\big\|_{L_p(\Omega)} & \lesssim |t-s|^{1/2} \sup_{r \in [s, t]} \| Y^n_r\|_{L_p(\Omega)}
\\
& \lesssim  |t-s|^{1/2}\big(|T-S|^{1/2} \|  Y^n  \|_{C^{1/2}([S, T]; L_p(\Omega))} +\| Y^n_S\|_{L_p(\Omega)}\big).
\end{equs}
Similarly, for the last term we have 
\begin{equs}
\big\| \int_s^t [\sigma (X^n_r)- \sigma (X^n_{\kappa_n(r)})] \, d B_r \big\|_{L_p(\Omega)} &  \lesssim  |t-s|^{1/2} \sup_{r \in [s, t]} \| X^n_r-X^n_{\kappa_n(r)}\|_{L_p(\Omega)}
\\
&\lesssim  |t-s|^{1/2} n^{-1/2}. 
\end{equs}
Hence, combining all the above estimates with \eqref{eq:error-decomp}, gives 
\begin{equs}
\|Y^n_t-Y^n_s\|_{L_p(\Omega)} \lesssim  |t-s|^{1/2}\big(|T-S|^{1/2} \|  Y^n  \|_{C^{1/2}([S, T]; L_p(\Omega))} +\| Y^n_S\|_{L_p(\Omega)}+n^{-1/2}\big).
\end{equs}
Dividing  by $|t-s|^{1/2}$ and taking suprema over $s, t \in [S,T]$, gives 
$$
[ Y^n] _{C^{1/2}([S, T]; L_p(\Omega))} \lesssim \left(    |T-S| ^{1/2} \|  Y^n  \|_{C^{1/2}([S, T]; L_p(\Omega))} +\| Y^n_S\|_{L_p(\Omega)} + n^{-1/2}   \right),
$$
which further gives 
$$
\|Y^n \| _{C^{1/2}([S, T]; L_p(\Omega))} \leq N_0 \left(    |T-S| ^{1/2} \|  Y^n  \|_{C^{1/2}([S, T]; L_p(\Omega))} +\| Y^n_S\|_{L_p(\Omega)} + n^{-1/2}   \right)
$$
with $N_0$ not depending on $S,T,n$.
If $|T-S| < 1/(2N_0)$, the above gives 
\begin{equs}
\|Y^n \| _{C^{1/2}([S, T]; L_p(\Omega))} \lesssim \left(   \| Y^n_S\|_{L_p(\Omega)} + n^{-1/2}   \right).
\end{equs}
Set now $S_k= k/   \left \lceil{2 N_0}\right \rceil  $, $k =0, ..., \left \lceil{2 N_0}\right \rceil$. We obviously have $|S_{k+1}-S_k| \leq 1/(2N_0)$, which gives
\begin{equs}      \label{eq:to-be-substituted}
\|Y^n \| _{C^{1/2}([S_k, S_{k+1}]; L_p(\Omega))} \lesssim \left(   \| Y^n_{S_k}\|_{L_p(\Omega)} + n^{-1/2}   \right),
\end{equs}
for all $k =0, ..., \left \lceil{2 N_0}\right \rceil$. In particular, this gives 
\begin{equs}
 \| Y^n_{S_{k+1}}\|_{L_p(\Omega)}  \lesssim  \left(   \| Y^n_{S_k}\|_{L_p(\Omega)} + n^{-1/2}   \right),
\end{equs}
which by iteration gives 
\begin{equs}
 \| Y^n_{S_k}\|_{L_p(\Omega)}  \lesssim n^{-1/2}.
\end{equs}
Substituting the above into \eqref{eq:to-be-substituted}, we conclude that 
\begin{equs}
\|Y^n \| _{C^{1/2}([S_k, S_{k+1}]; L_p(\Omega))}\lesssim n^{-1/2} .
\end{equs}
Consequently, we get 
\begin{equs}
\|Y^n \| _{C^{1/2}([0, 1]; L_p(\Omega))} &  \lesssim  \sup_k \|Y^n \| _{C^{1/2}([S_k, S_{k+1}]; L_p(\Omega))}
 \lesssim n^{-1/2},
\end{equs}
which shows the first inequality of the claim. The second now follows from \cref{lem:interchange-LpCa}.
\end{proof}

\begin{lemma}\label{lem:limit-singlar-term}
Let Assumptions \ref{asn:sigma} and \ref{asn:Holder} hold. For $n \in \mathbb{N}$,  let $\{ ( \bar \Omega,  \bar \cF,   \bar\bP),  \bar{\mathbb{F}}^n, \bar{\mathcal{X}}^n,  \bar B^n\}$,  $\{ ( \bar \Omega,  \bar \cF,   \bar\bP),  \bar{\mathbb{F}}, \bar{\mathcal{X}},   \bar B\}$, be  solutions of \eqref{sde}.  Define the processes
\begin{equs}
d\bar X^n_t&= b(\bar X^n_{\kappa_n(t)}) \, dt + \sigma(\bar X^n_{\kappa_n(t)}) \, d\bar B^n_t,  \qquad \bar X^n_0=x_0,
\\
\bar V^n& = \sqrt{n}( \bar{\mathcal{X}}^n- \bar X^n).
\end{equs}
Finally,  let  $\beta \in (0, 1/2]$, $q,  \bar{q}\geq 1$ such that  $(1+\alpha)/2+\beta >1$ and $\bar q<q$, and assume that 
\begin{enumerate}[(i)]
\item  There exists $\bar V \in L_q(\bar \Omega ; C^\beta([0, 1]; \R^d))$ such that $lim_{n \to \infty} \|\bar V-\bar V^n\|_{ L_q(\bar\Omega; C^\beta([0,1])) }=0$,

\item   $lim_{n \to \infty} \sup_{t \in [0, 1]} \|\bar{\mathcal{X}}^n_t -\bar{\mathcal{X}}_t\|_{ L_{q\bar q/(q-\bar q)}(\bar\Omega)}=0$.
\end{enumerate}

Then,  for any $f \in C^{\alpha+}(\R^d)$  and  $\gamma \in (0,  \alpha)$ we have 
\begin{equs}         \label{eq:convergence-singular-term}
\lim_{n \to \infty} \Big\|\int_0^\cdot \sqrt{n}  (f(\bar{\mathcal{X}}^n_r)-f(\bar X^n_r))\, dr -  \int_0^ \cdot \bar  V^{(j)}_r \, d( Q^{\bar{\mathcal{X}}}f)^{(j)}_r  \Big\|_{C^{(1+\gamma)/2} ([0, 1] ; L_{\bar q}( \bar \Omega))}  =0 .
\end{equs}
\end{lemma}

\begin{proof}
Let $\eps \in (0,1)$ be small so that  
\begin{equs}
\theta+\beta:= (1+\alpha)/2-\eps  + \beta >1,\qquad  (1+\gamma)/2 < \theta.
\end{equs}  
Throughout the proof the proportionality constants in $\lesssim$ depend only  on $\alpha,\beta,\gamma,\eps,d,q,\bar q,$ and $\|\bar V\|_{L_q(\bar\Omega;C^\beta([0,1]))}$. 
Notice that by \cref{rem:Holder-reg-a.s.},  for $f\in C^{\alpha+}$ we have that   $ Q^{\bar{\mathcal{X}}}f \in L_p ( \bar  \Omega ; C^\theta ([0,1]; \R^d))$ for all $p \geq 1$.  In particular, the Young integrals appearing in 
\eqref{eq:convergence-singular-term} are well defined and satisfy
\begin{equs}
 \big\|  \int_s^t \bar{V}^{(j)}_r \, d ( Q^{\bar{\mathcal{X}}}f)^{(j)}_r  \big\|_{L_{\bar{q}}(\bar \Omega)} &  \lesssim \big\|  \| \bar{V} \|_{C^\beta([0,1])}  \| Q^{\bar{\mathcal{X}}}f \|_{C^\theta ([0,1])}  \big\|_{L_{\bar{q}}(\bar \Omega)} |t-s|^\theta 
 \\
&   \lesssim \| \bar{V} \| _{L_q(\bar \Omega ; C^\beta([0,1]) ) }   \|  Q^{\bar{\mathcal{X}}}f \|_{L_{q\bar q/(q-\bar q)}( \bar \Omega; C^{ \theta }([0,1]) ) } |t-s|^\theta 
\\
& \lesssim \| f\|_{C^{\alpha}(\R^d)} |t-s|^\theta,
\end{equs}
where we have used again \cref{rem:Holder-reg-a.s.}. 
Hence,  the mapping $\cI$ given by 
\begin{equs}
f  \mapsto \sum_{j=1}^d \int_0^\cdot \bar{V}^{(j)}_r \, d  (  Q^{\bar{\mathcal{X}}}f)^{(j)}_r
\end{equs}
is a bounded linear operator from $ C^{\alpha+} ( \R^d)$ to $ C^\theta( [0,1]; L_{\bar q} (\bar \Omega))$.
Let us similarly define the operators 
 $\mathcal{I}^n:C^{\alpha+} ( \R^d) \to  C^\theta( [0,1]; L_{\bar q} (\bar \Omega))$ by 
 \begin{equs}
 f \mapsto \int_0^\cdot  \sqrt{n}    (f(\bar{ \mathcal{X}}^n_r)-f(\bar X^n_r))\, dr. 
 \end{equs}
By applying  \cite[Lemma 6.4]{ButDarGEr} with  $\tau=\beta$  we get 
 \begin{equs}
\big\|  \int_s^t  \sqrt{n}    (f(\bar{ \mathcal{X}}^n_r)-f(\bar X^n_r))\, dr \big\|_{L_{\bar q}(\bar \Omega)} \lesssim\| V^n \|_{ C^\beta ([0, 1]; L_{\bar q}(\bar \Omega) ) } \| f \|_{C^{\alpha}(\R^d)} |t-s|^\theta 
 \end{equs}
By assumption $(i)$ and the trivial embedding $ L_{\bar q}(\bar \Omega; C^\beta ([0, 1]))\subset C^\beta ([0, 1]; L_{\bar q}(\bar \Omega) )$ we get that
 \begin{equs}   \label{eq:Op-norm-bounded}
 \sup_{n\in\N} \|\mathcal{I}^n\|_{\cL(C^{\alpha+} ( \R^d);  C^\theta( [0,1]; L_{\bar q} (\bar \Omega))}\lesssim 1.
 \end{equs}
If $g  \in C^\infty(\R^d)$,  by  \cref{prop:integral_operator} \eqref{item:op-on-smooth} and the fundamental theorem of calculus  we have 
 \begin{equs}
 \| &(\mathcal{I}^n_t(g)- \mathcal{I}^n_s(g))-( \mathcal{I}_t(g)- \mathcal{I}_s(g))\| _{L_{\bar q}(\bar \Omega)}
 \\
  & =\  \big\|  \int_s^t \big(  \int_0^1 \nabla g (\bar{\mathcal{X}}^n_r + \vartheta ( \bar X^n_r -\bar{\mathcal{X}}^n_r) ) \bar V^n_r  \, d \vartheta - \nabla g(\bar{\mathcal{X}}_r) \bar V_r  \big) \, dr  \big\|_{L_{\bar q}(\bar \Omega)}
\\
  &  \leq \|g\|_{C^2}|t-s|   \Big(  \sup_{t \in [0, 1]} \|\bar V^n_t-\bar V_t \|_{L_{\bar q}(\bar \Omega)}+  \sup_{t \in [0, 1]} \| (\bar{\mathcal{X}}^n_t-\bar{\mathcal{X}}_t) \bar V_t  \|_{L_{\bar q}(\bar \Omega)} 
\\
&\,\hspace{7cm}+ \sup_{t \in [0, 1]} \| (\bar X^n_t-\bar{\mathcal{X}}^n_t) \bar V_t  \|_{L_{\bar q}(\bar \Omega)} \Big)
 \\
&  \leq \|g\|_{C^2}|t-s|   \Big(  \|V-V^n\|_{ L_q(\bar\Omega; C^\beta([0,1])) } 
 \\
 & \quad + \big(  \sup_{t \in [0, 1]} \|(\bar{\mathcal{X}}^n_t-\bar{\mathcal{X}}_t)  \|_{L_{q\bar q/(q-\bar q)}(\bar \Omega)}      + \sup_{t \in [0, 1]} \|\bar X^n_t-\bar{\mathcal{X}}^n_t  \|_{L_{q\bar q/(q-\bar q)}(\bar \Omega)}      \big) \| V\|_{C^\beta([0,1]; L_q(\bar\Omega) )} \Big).
 \end{equs}
 From  $(i)$ and $(ii)$ of our assumptions combined with  \cite[Theorem 2.1]{ButDarGEr}, 
we conclude that  for $g \in C^\infty(\R^d)$
 \begin{equs}
 \lim_{n \to \infty} \| \mathcal{I}_n(g)- \mathcal{I}(g)\|_{C^\theta ([0, 1] ; L_{\bar q}( \bar \Omega))} =0. 
 \end{equs}
 Since, $f \in C^{\alpha+}(\R^d)$,  for any $\delta >0$,  we can find $\tilde{f} \
 \in C^\infty(\R^d)$ such that $\| f -\tilde{f} \|_{C^{\alpha}(\R^d)}\leq \delta$.  Then, we see that 
 \begin{equs}
  &\| \mathcal{I}_n(f)- \mathcal{I}(f)\|_{C^\theta ([0, 1] ; L_{\bar q}( \bar \Omega))} 
  \\& \leq  \| \mathcal{I}_n(f)- \mathcal{I}_n(\tilde{f})\|_{C^\theta ([0, 1] ; L_{\bar q}( \bar \Omega))} +  \| \mathcal{I}_n(\tilde{f})- \mathcal{I}(\tilde{f})\|_{C^\theta ([0, 1] ; L_{\bar q}( \bar \Omega))} 
  \\
 & \quad +   \| \mathcal{I}(f)- \mathcal{I}(\tilde{f})\|_{C^\theta ([0, 1] ; L_{\bar q}( \bar \Omega))} 
 \\
 &  \leq ( \| \mathcal{I}_n\| + \| \mathcal{I}\|) \| f -\tilde{f} \|_{C^{\alpha}(\R^d)}+ \| \mathcal{I}_n(\tilde{f})- \mathcal{I}(\tilde{f})\|_{C^\theta ([0, 1] ; L_{\bar q}( \bar \Omega))},
 \end{equs}
which  implies that $ \limsup_{n \to \infty}\| \mathcal{I}_n(f)- \mathcal{I}(f)\|_{C^\theta ([0, 1] ; L_{\bar q}( \bar \Omega))}  \lesssim \delta$, and since $\delta>0$ was arbitrary and $\theta> (1+\gamma)/2$,  the claim is proved. 
\end{proof}

The following proposition is well known and it follows by \cite[Theorem 5.1]{Jacod_Protter}, for example.

\begin{proposition}         \label{prop:convergence-BMs}
Let $(\Omega, \cF, \bP)$ be a probability space carrying an $\R^d$-valued Brownian motion $B$ and define $W^n$ as in \cref{thm:main_Holder}. There exists an extension $(\tilde{\Omega}, \tilde{\cF}, \tilde{\bP})$ of $(\Omega, \cF, \bP)$, carrying a $\R^{d \times d} $-valued standard Brownian motion $W$ independent of $B$, such that 
the laws of $(W^n,B)$ converge weakly to the law of  $(W,B)$ on $C([0, 1];\R^d\times\R^{d \times d})$.
\end{proposition}

\begin{proof}[Proof of Theorems \ref{thm:existence_uniqueness} and \ref{thm:main_Holder}]

The uniqueness part of \cref{thm:existence_uniqueness} follows from \cref{prop:weak_uniqueness}.
	Let us set $S^n=(X, V^n, B, W^n)$. We start by showing that any  subsequence $\{S^n\}_{n \in \mathbb{N}_1}  \subset  \{S^n\}_{n \in \mathbb{N}}$, there exists a further subsequence $\{S^n\}_{n \in \mathbb{N}_2}  \subset  \{S^n\}_{n \in \mathbb{N}_1}$ which converges weakly to a  solution of the system \eqref{sde}, \eqref{eq:V}. This will verify the existence part of \cref{thm:existence_uniqueness} and since $\{S^n\}_{n \in \mathbb{N}_1}$ is arbitrary, will also prove \cref{thm:main_Holder}.

	Let $ \{S_n\}_{n \in \mathbb{N}_1} $ be an arbitrary subsequence. Let us fix $\beta \in (0, 1/2)$ such that $(1+\alpha)/2+\beta>1$. Moreover, we fix  $\gamma, \gamma' \in (0, 1/2)$ such that  $1/2 >\gamma > \gamma'> \beta  $, and we chose $q>1$ large enough such that 
	\begin{equs}       \label{eq:embeddings}
	C^{\gamma}([0,1])  \subset \subset  W^{\gamma'}_q ([0, 1]) \subset \subset  C^\beta([0,1]).
	\end{equs}
	For any  $p \geq 1$, by Lemma \ref{lem:uniform-bounds}, we have that 
	\begin{equs} \label{eq:uniform-V-n}
	\sup_{n \in \mathbb{N}_1} \| V^n\|_{L_p(\Omega, C^\gamma([0,1]; \R^d))} < \infty.
	\end{equs} 
	Moreover,  by the boundedness of the coefficients and the  BDG inequality  one can easily see that $(X, X^n, B, W^n)$ is bounded in $L_p(\Omega; \mathcal{C}^\gamma) $, where 
 \begin{equs}
     \mathcal{C}^\gamma=  \big( C^\gamma([0, 1]; \R^d) \big)^3 \times C^\gamma([0, 1]; \R^{d\times d}).
 \end{equs}
	By the first compact embedding in \eqref{eq:embeddings}, we have that the laws of $\{(X, X^n, V^n, B, W^n)\}_{n \in \mathbb{N}_1} $ are tight on 
 \begin{equs}
 \mathcal{W}^{\gamma'}_q=\big(W^{\gamma'}_q ([0, 1]; \R^d ) \big)^4 \times W^{\gamma'}_q ([0, 1]; \R^{d \times d} ) .
 \end{equs}The latter  is a Polish space and, therefore, 
	by Prokorov's theorem, we have that the laws of a further subsequence  $\{(X, X^n, V^n, B, W^n)\}_{n \in \mathbb{N}_2}$ 
	converge weakly.  By Skorohod's representation theorem, there exists a probability space $( \hat \Omega, \hat \cF, \hat  \bP) $ and random variables 
	\begin{equs}
	(\hat{ \mathcal{X}}^{n},  \hat  X^n, \hat   V^n, \hat  B^n , \hat W^n)  : \hat  \Omega \to  \mathcal{W}^{\gamma'}_q
	\end{equs}
	for $n \in \mathbb{N}_2$, 
	and 
	\begin{equs}
	(\hat{ \mathcal{X}}, \hat  X,  \hat   V, \hat  B, \hat W)   : \hat \Omega \to  \mathcal{W}^{\gamma'}_q, 
	\end{equs}
	such that 
	\begin{equs}      \label{eq:same-distribution}
	(\hat{ \mathcal{X}}^{n},  \hat  X^n, \hat   V^n, \hat  B^n , \hat W^n) \overset{d}{=}  (X, X^n, V^n, B, W^n) \ \text{ on  \   $ \mathcal{W}^{\gamma'}_q$}, 
	\end{equs}
	and $\hat \bP$-a.s. 
	\begin{equs}            \label{eq:almost-sure-convergence}
	(\hat{ \mathcal{X}}^{n},  \hat  X^n, \hat   V^n, \hat  B^n , \hat W^n) \to (\hat{ \mathcal{X}}, \hat  X,  \hat   V, \hat  B, \hat W)   \ \text{ as \ $n \to \infty$},
	\end{equs}
	in $ \mathcal{W}^{\gamma'}_q$. 
	Let $\{\hat \cF^n_t\}_{t \geq0}$ be the augmentation of the filtration generated by $\hat{ \mathcal{X}}^n, \hat B^n$, and $\hat W^n$.  It is easy to check that $ \hat B^n$  is a $\{\hat \cF^n_t\}_{t \geq0}$-Brownian motion. 
	Notice that by \eqref{eq:same-distribution} and the continuity of the coefficients, the following equalities follow 
	\begin{equs}        
	d \hat{ \mathcal{X}}^n_t & = b( \hat{ \mathcal{X}}^n_t) \, dt + \sigma ( \hat{ \mathcal{X}}^n_t) \, d \hat B^n_t, \qquad X^{(n)}_0=x_0    \label{eq:idk1}
	\\
	d \hat X^n_t  & = b( \hat  X^n_{\kappa_n(t)} ) \, dt + \sigma ( \hat X^n_{\kappa_n(t)} ) \, d \hat B^n_t, \qquad X^{(n)}_0=x_0       \label{eq:idk2}
	\\
	\hat V^n &= \sqrt{n}(\hat{ \mathcal{X}}^n_t- \hat X^n_t )          \label{eq:idk3}
	\\
	\hat W^n_t &= \sqrt{2 n} \int_0^t (  \hat B^n_r -\hat B^n_{\kappa_n(r)} )  \, d \hat B_r.      \label{eq:idk4}
	\end{equs}
By \eqref{eq:same-distribution} and the fact that $ (X, X^n, V^n, B, W^n)$ are bounded in $L_p(\Omega;  \mathcal{W}^{\gamma'}_q)$  for any $p \geq 1$,  we get that 
	\begin{equs}           \label{eq:uniform_integrability}
	\sup_{n \in \mathbb{N}_2}\|(\hat{ \mathcal{X}}^{n},  \hat  X^n, \hat   V^n, \hat  B^n , \hat W^n)\|_{L_p(\hat \Omega;  \mathcal{W}^{\gamma'}_q)}< \infty,
	\end{equs}
	for any  $p\geq 1$. 
	The almost sure convergence from \eqref{eq:almost-sure-convergence} combined with the uniform integrability obtained by \eqref{eq:uniform_integrability} gives by Vitali's theorem that for all $p\geq 1$
	\begin{equs}           \label{eq:convergence-in-S'}
	\lim_{n \to 
	\infty}\|(\hat{ \mathcal{X}}^{n},  \hat  X^n, \hat   V^n, \hat  B^n , \hat W^n)- (\hat{ \mathcal{X}}, \hat  X,  \hat   V, \hat  B, \hat W)\|_{L_p(\hat \Omega;  \mathcal{W}^{\gamma'}_q)}=0.
	\end{equs}
	We claim that $\hat{ \mathcal{X}} =\hat{X}.$
	Indeed,  by \eqref{eq:convergence-in-S'} we have that $\hat{ \mathcal{X}}^n \to \hat{ \mathcal{X}}$ and $\hat{X}^n \to \hat{X}$ while  \eqref{eq:idk1}-\eqref{eq:idk2} and \cite[Theorem 2.7]{ButDarGEr} imply that $\hat{ \mathcal{X}}^n -\hat{X}^n \to 0$, with all three limits understood,  for example,  in $C([0,1];\R^d)$ in probability (recall the second embedding in \eqref{eq:embeddings}). 
 
 Next,  let $\hat{\mathbb{F}}:= \{\hat \cF_t\}_{t \in [0, 1]}$ be the augmentation of the filtration generated by $\hat{X}, \hat{V},  \hat B, \hat W$.   We claim that 
	$\{ ( \hat \Omega,  \hat  \cF,   \hat{\mathbb{F}}, \hat \bP), \hat{ X}, \hat V, \hat B, \hat W \}$ is a solution of the system  \eqref{sde}, \eqref{eq:V}. 
 First of all,  by using \eqref{eq:same-distribution}, \eqref{eq:almost-sure-convergence}, and  \cref{prop:convergence-BMs},  it follows that $ \hat B$ and $\hat W$ are independent Brownian motions, that $\hat B$ is an $\hat{\mathbb{F}}$-Brownian motion, and that $\hat W$ is an $\hat{\mathbb{F}}$-martingale. This is enough to conclude that $\hat W$ is also an $\hat{\mathbb{F}}$-Brownian motion.
 By the convergence in \eqref{eq:almost-sure-convergence} combined with the continuity of the coefficients and \cref{lem:convergence_stochastic_integrals}, we can let $n \to \infty$ in \eqref{eq:idk1} to obtain that  $\{(\hat{\Omega}, \hat{\cF}, \hat{\mathbb{F}}, \hat{\bP}), \hat{ \mathcal{X}}, \hat B\}$ is a solution of \eqref{sde}. Moreover, since pathwise uniqueness holds for \eqref{sde}, it follows that $\mathcal{X}$ is adapted to the augmentation of the filtration generated by $\hat{B}$. In particular, this shows that $( \mathcal{X}, \hat B) \indep \hat{W}$, since $ \hat B \indep \hat{W}$ .
	Consequently, \eqref{item:def_sol_1} and \eqref{item:def_sol_2} from  \cref{def:solution_of_system} are satisfied. Moreover, \eqref{item:def_sol_3} follows from the fact that $\hat V \in W^{\gamma'}_q ([0, 1]; \R^d)$  and the second embedding in \eqref{eq:embeddings}. Finally, we focus on showing \eqref{item:def_sol_4}, that is, on showing that $\hat V=(\hat V^{(1)},\ldots,\hat V^{(d)})$ satisfies
	\begin{equs}
			d \hat V_t=    \hat V _t d ( Q^Xb)_t + \nabla \sigma( \hat X_r) \hat V_rd \hat B_r+\frac1{\sqrt2}\sigma \nabla \sigma ( \hat X_r)d  \hat W_r.
			 \label{eq:after_limit}
		\end{equs}
	From \eqref{eq:idk1}-\eqref{eq:idk2},  we have  
	\begin{equs}
	 \hat V^n_t&  =  \sqrt{n} \int_0^t \big(  b (\hat{ \mathcal{X}}^{n}_s)- b ( \hat  X^n_{\kappa_n(s)}) \big) \, ds
	 + \sqrt{n} \int_0^t \big(  \sigma (\hat{ \mathcal{X}}^{n}_s)- \sigma ( \hat  X^n_{\kappa_n(s)}) \big) \, d \hat B^n_s                    \label{eq:before_limit}
	\\
	& =I^n_t+E^n_t+\bar I^n_t+\bar E^n_t,  
	\end{equs}
	where 
			\begin{equs}
				I^n_t & =\int_0^t \sqrt n(b( \hat{ \mathcal{X}}^n_r)-b (\hat X^n_r))dr,  
				\\
				 E^n_t & =\int_0^t\sqrt n(  b(\hat X^n_r)-b (\hat X^n_{\kappa_n(r)}))dr,
				\\
				\bar I^n_t & =  \int_0^t\sqrt n (\sigma( \hat{ \mathcal{X}}^n_r)-\sigma(\hat X^n_r))d\hat B^n_r, 
				\\
				 \bar E^{n}_t& =\int_0^t\sqrt n(\sigma (\hat X^n_r)-\sigma (\hat X^n_{\kappa_n(r)}))d \hat B^n_r.
			\end{equs}
	By \eqref{eq:convergence-in-S'} and \eqref{eq:embeddings}, we see that the conditions of \cref{lem:limit-singlar-term} are satisfied, and \cref{lem:limit-singlar-term} combined with  \cref{lem:interchange-LpCa}  implies that
	\begin{equs}        \label{eq:conv_In}
	\big\| I^n- \int_0^ \cdot \hat V_r \,  d (Q^{\hat X}b)_r \big\|_{C^\beta([0,1])}  \xrightarrow{\hat{\bP}} 0, 
	\end{equs}
	as $n \to \infty$. 	Here and below 	$\xrightarrow{\hat{\bP}}$ denotes convergence in probability.
			In addition,  by \cref{lem:quad1} and \cref{lem:interchange-LpCa} we have  
	\begin{equs}		  \label{eq:conv_En}
	 \|E^{n}\|_{C^\beta([0,1])} \xrightarrow{\hat{\bP}} 0,  
	\end{equs}
	as $n \to \infty$.
	Moving to $\bar I^n$,  by the fundamental theorem of calculus we have that 
	\begin{equs}
	 \bar I^n_t &= \int_0^t\sqrt n (\sigma ( \hat{ \mathcal{X}}^n_r)-\sigma (\hat X^n_r))d\hat B^n_r
	 \\
	 &=  \int_0^t    \hat{V}^n_r \int_0^1 \nabla\sigma( \hat{ \mathcal{X}}^n_r+ \vartheta( \hat X^n_r-\hat{ \mathcal{X}}^n_r)) \, d \vartheta \,d\hat B^n_r.
	 \end{equs} 
	 By \eqref{eq:convergence-in-S'} and the fact that $
	 \sigma \in C^2$,   we get that 
	 \begin{equs}
	 \big\| \hat{V}^n\int_0^1 \nabla \sigma ( \hat{ \mathcal{X}}^n + \vartheta( \hat X^n-\hat{ \mathcal{X}}^n)) \, d \vartheta- \hat V \nabla \sigma (\hat {X})\big\|_{C([0,1])}  \xrightarrow{\hat{\bP}} 0,
	\end{equs}  
	and
	\begin{equs}
	\|\hat B^n-\hat B\|_{C([0,1])}  \xrightarrow{\hat{\bP}} 0,
	\end{equs}
	as $n \to \infty$.  Consequently, by  \cref{lem:convergence_stochastic_integrals}, we get that 
	\begin{equs}          \label{eq:conv_bar_In}
	\big\|\bar I^n - \int_0^ \cdot  \hat V_r \nabla  \sigma(\hat {X}_r ) \, d \hat{B}_r\|_{C([0,1])}  \xrightarrow{\hat{\bP}} 0,
	\end{equs}
	as $n \to \infty$. Finally, we deal with $\bar E^n$. By the fundamental theorem of calculus again, we have 
	\begin{equs}
	\bar E^n_t & =  \int_0^t \sqrt{n} ( \hat{X}^n_r-  \hat{X}^n_{\kappa_n(r)} ) \int_0^1 \nabla \sigma\big( \hat{X}^n_r+ \vartheta ( \hat{X}^n_{\kappa_n(r)}- \hat{X}^n_r) \big) \, d \vartheta \, d \hat B^n_r
	\\
	& =   \int_0^t \sqrt{n} (r-\kappa_n(r)) b( \hat{X}^n_{\kappa_n(r)})  \int_0^1 \nabla \sigma\big( \hat{X}^n_r+ \vartheta ( \hat{X}^n_{\kappa_n(r)}- \hat{X}^n_r) \big) \, d \vartheta \, d \hat B^n_r
	\\
	&\qquad+ \frac{1}{\sqrt{2}}  \int_0^t \sigma( \hat{X}^n_{\kappa_n(r)})  \int_0^1 \nabla \sigma\big( \hat{X}^n_r+ \vartheta ( \hat{X}^n_{\kappa_n(r)}- \hat{X}^n_r) \big) \, d \vartheta \, d \hat{W}^n_r
	\\
	&=: \bar E^{n, 1}_t + \bar E^{n, 2}_t.
	\end{equs}
	Since $(r-\kappa_n(r))$ is of order $n^{-1}$ and since $b$ and $\nabla \sigma$ are bounded,  it follows by Davis' inequality that  $\|\bar E^{n, 1} \|_{C([0,1])}  \xrightarrow{\hat{\bP}} 0$. Next, by \eqref{eq:convergence-in-S'} it follows that we have 
	\begin{equs}
	\sup_{r \in [0,1]}|\sigma( \hat{X}^n_{\kappa_n(r)})  \int_0^1 \nabla \sigma\big( \hat{X}^n_r+ \vartheta ( \hat{X}^n_{\kappa_n(r)}- \hat{X}^n_r) \big) \, d \vartheta -
	\sigma(\hat{X}_r)  \nabla \sigma( \hat{X}_r)|  \xrightarrow{\hat{\bP}} 0,
	\end{equs}
	and 
	\begin{equs}
	\|\hat W^n-\hat W
 \|_{C([0,1])}  \xrightarrow{\hat{\bP}} 0,
	\end{equs}
	as $n \to \infty$.
	Consequently, by \cref{lem:convergence_stochastic_integrals}, we get that 
	\begin{equs}          \label{eq:conv_bar_En}
	\big\|\bar E^n -\frac{1}{\sqrt{2}}\int_0^ \cdot  \sigma (\hat{X}_r)  \nabla \sigma ( \hat{X}_r)\, d \hat{W}_r\|_{C([0,1])}  \xrightarrow{\hat{\bP}} 0,
	\end{equs}
	as $n \to \infty$.
	Recall that almost surely $\|\hat{V}^n-\hat{V}\|_{C([0,1])}\to 0$,  which combined with \eqref{eq:conv_In}-\eqref{eq:conv_bar_En}, upon letting $n \to \infty$ in \eqref{eq:before_limit} leads to \eqref{eq:after_limit}. Hence, this shows that \eqref{item:def_sol_4} from \cref{def:solution_of_system} is also satisfied. This finishes the proof.

\end{proof}

\section{Proof of main results for Sobolev $b$}     \label{sec:proofs_Sobolev} 
\subsection{Tightness}
In this section we focus on the following tightness result.
\label{sub:tightness}
	\begin{theorem}\label{prop.tight.sobolev}
		Let Assumptions \ref{asn:sigma} and \ref{asn:Sobolev} hold.  For any $p\geq1$ and $\gamma \in (0, 1)$ we have 
  \begin{equs}
      \sup_{n\in\N} \|  V^n \|_{C^{\gamma/2}([0, 1]; L_p(\Omega))} < \infty.
  \end{equs}
 In particular, for any $\gamma \in (0, 1)$, the laws of $\{V^n\}_{n \in \mathbb{N}}$ are tight on $C^{\gamma/2}([0, 1])$.
	\end{theorem}

Before the proof of the above theorem, we need some auxiliary results. Let $u$ be the solution of \eqref{eqn.pde}, for a $\theta>0$ to be determined later on.
By applying It\^o's formula for $u(t, X_t)$, we obtain that
		\begin{equs}\label{eq-bX}
		\int_0^t b(X_r) d r = u(0,X_0)-u(t, X_t) + \theta \int_0^t u(r, X_r) d r
		 + \int_0^t (\partial_i u \sigma^{(i,j)})(r, X_r) d B^{(j)}_r,
		\end{equs}
		and similarly,
		\begin{equs}\label{eq-bXn}
		\int_0^t b(X^n_r) d r &= u(0,X^n_0)-u(t, X^n_t) + \theta \int_0^t u(r, X^n_r) d r
		\nonumber\\&\quad+\int_0^t \partial_i u(r, X_r^n) [b^{(i)}( X_{\kappa_n(r)}^n)-b^{(i)}(X^n_r)]  d r
		\nonumber\\&\quad+\int_0^t \partial_{ij}^2 u(r,X^n_r)[a^{(i,j)}(X^n_{\kappa_n(r)})-a^{(i,j)}(X^n_r)]dr
		\nonumber\\&\quad+\int_0^t \partial_i u(r,X^n_r)\sigma^{(i,j)}(X^n_{\kappa_n(r)})dB^{(j)}_r.
		\end{equs}
		From equations \eqref{sde} and \eqref{EM}, we have
		\begin{equs}
			X_t-X^n_t
			&=x_0-x_0^n
			+\int_0^t[b(X_r)-b(X^n_r)]dr
			+\int_0^t[b(X^n_r)-b(X^n_{\kappa_n(r)})]dr
			\\&\quad+\int_0^t[\sigma(X_r)-\sigma(X^n_r)]dB_r
			+\int_0^t[\sigma(X^n_r)-\sigma(X^n_{\kappa_n(r)})]dB_r.
		\end{equs}
		We plug \eqref{eq-bX} and \eqref{eq-bXn} into the identity above  to obtain that
		\begin{equation}\label{eqn.transformSDE}
				X_t-X^n_t
				=x_0-x^n_0+\sum_{k=0}^6  J^k_{t}
		\end{equation}
		where for $\ell \in \{1, \ldots, d\}$
		\begin{equs}
  \begin{aligned}   \label{eq:def_J}
			J^{0,(\ell)}_t &=u^{(\ell)}(t,X_t)-u^{(\ell)}(t,X^n_t)
			\\J^{1,(\ell)}_t&=\theta\int_0^t (u^{(\ell)}(r,X_r)-u^{(\ell)}(r,X^n_r))dr
			\\J^{2,(\ell)}_t&=\int_0^t[I^{(\ell,i)}+\partial_i u^{(\ell)}(r,X^n_r)][b^{(i)}(X^n_r)-b^{(i)}(X^n_{\kappa_n(r)})]dr
			\\J^{3,(\ell)}_t&=\int_0^t \partial_{i j}^2 u^{(\ell)}(r,X^n_r)[a^{(i,j)}(X^n_r)-a^{(i,j)}(X^n_{\kappa_n(r)})]dr
			\\J^{4,(\ell)}_t&=\int_0^t[I^{(\ell,i)}+\partial_i u^{(\ell)}(r,X^n_r)][\sigma^{(i,j)}(X_r)-\sigma^{(i,j)}(X^n_r)]dB^{(j)}_r
			\\J^{5,(\ell)}_t&=\int_0^t[I^{(\ell,i)}+\partial_i u^{(\ell)}(r,X^n_r)][\sigma^{(i,j)}(X^n_r)-\sigma^{(i,j)}(X^n_{\kappa_n(r)})]dB^{(j)}_r
			\\J^{6,(\ell)}_t&=\int_0^t[\partial_i u^{(\ell)}(r,X_r)-\partial_i u^{(\ell)}(r,X^n_r)] \sigma^{(i,j)}(X_r)dB^{(j)}_r.
   \end{aligned}
		\end{equs}
    Finally, let us introduce the process
	\begin{equation}\label{def.Ant}
		A_t^n:= t+\int_0^t \left[ \cmm|\nabla^2 u|(s,X_s)+\cmm|\nabla^2 u|(s,X^n_s)\right]^2 d s.
	\end{equation}
where $\mathcal{M}$ is the Hardy--Littlewood maximal operator (in the spatial variable) defined as  
	\begin{equs}
		\cmm f(x):=\sup_{0<r<\infty}\frac{1}{|\cB_r|}\int_{\cB_r}f(x+y)dy, \quad \cB_r:=\{x\in\mathbb{R}^d:|x|<r\},\quad r>0.
	\end{equs}
It is well known that there exists a constant $N$ depedning only on $d$  such that for all $f \in W^1_{1, loc}(\R^d)$, for almost every $x, y \in \R^d$, we have
\begin{equs}        \label{eq:Lip_Max}
    |f(x)-f(y)| \leq N |x-y|\big( \mathcal{M}|\nabla f| (x) +  \mathcal{M}|\nabla f| (y) \big), 
\end{equs}
and that for all $f \in L_p(\R^d)$, $p \in (1, \infty)$, we have 
\begin{equs}   \label{eq:Hardy-Littlewood} 
    \|  \mathcal{M}f \|_{L_p(\R^d)} \leq N \|f \|_{L_p(
    R^d)}. 
\end{equs}
 The following is a straightforward consequence of the  boundedness of the coefficients and the BDG inequality. 
	\begin{lemma}\label{lem.Xtkt}
		Let $X^n$ be defined by  \eqref{EM} and assume that $b$ and $\sigma$ are bounded. Then for every $p \in [1, \infty) $, there exists a constant $N=N(p, \|b\|_{C^0}, \| \sigma\|_{C^2})$, such that for all $n \in \mathbb{N}$ we have 
		\begin{equs}
			\sup_{t\in[0,1]}\|X^n_{t}-X^n_{\kappa_n(t)}\|_{L_p(\Omega)}\leq N n^{-1/2}.
		\end{equs}
	\end{lemma}
 For the proof of \cref{prop.weightedmoment} below, we will need the following Krylov-type estimate.
 \begin{lemma}               \label{lem:Krylov_est_Xn}
     Suppose that Assumption \ref{asn:sigma} holds and $b \in C^0(\R^d; \R^d)$ and let $p>(d+2)/2$. Let $X^n$ be defined by \eqref{EM}.  There exists a constant $N=(\| b\|_{C^0}, \|\sigma\|_{C^2}, \lambda, d,p)$ such that for all $g \in L_p((0, 1) \times \R^d)$ and all $(s, t) \in [0, 1]_{\leq}$ we have 
     \begin{equs}
         \E \int_s^t g(r,X^n_r) \, dr \leq N (t-s)^{1-(d+2)/(2p)} \|g\|_{L_p((0, 1) \times \R^d)}.
     \end{equs}
 \end{lemma}
 \begin{proof}
     It suffices to prove the claim for $g$ bounded with compact support. Let $\bar{X}^n$ be the Euler scheme corresponding to the driftless equation. From \cite[Lemma 5.5]{le2021taming} we have that 
     \begin{equs}  \label{eq:density_Xn}
         \E f(\bar{X}^n_r)  \lesssim r^{-d/2q} \| f\|_{L_q(\R^d)}, 
     \end{equs}
     for any function $f$ and any $ q \in [1, \infty)$.  Let us also denote by $\rho_n$ the Girsanov density, that is
     \begin{equs}
         \rho_n= \exp\Big( - \int_0^1 b \sigma^{-1} (X^n_{\kappa_n(s)}) dB_s-\frac{1}{2}\int_0^1 b \sigma^{-1} (X^n_{\kappa_n(s)}) dB_s \Big),
     \end{equs}
     and notice that due to the boundedness of $b$ and the uniform allipticity of $\sigma$, $\rho_n$ and $\rho_n^{-1}$ have finite moments of any order uniformly in $n\in \mathbb{N}$. Then, by H\"older's inequality and  \eqref{eq:density_Xn} with $q=2$, $f= |g|^{p/2}$, we have  
     \begin{equs}
         \E |g(r, X^n_r)| =   \E |g(r, X^n_r)| \rho^{2/p} \rho^{-2/p} & \lesssim \big(  \E |g(r, X^n_r)|^{p/2} \rho \big)^{2/p}= \big(  \E |g(r, \bar{X}^n_r)|^{p/2} \big)^{2/p}
         \\ & \lesssim \big(r^{-d/4} \| g(r, \cdot) \|_{L_p(\R^d)}^{p/2} \big)^{2/p} = r^{-d/2p} \| g(r, \cdot) \|_{L_p(\R^d)}.
     \end{equs}
     Hence, integrating the above inequality from $s$ to $t$ and applying H\"older's inequality proves the claim.
 \end{proof}
	\begin{proposition}\label{prop.weightedmoment}
		Let Assumptions \ref{asn:sigma} and \ref{asn:Sobolev} hold, and let $X$ and $X^n$ satisfy \eqref{sde} and \eqref{EM}, respectively. For every $p\in [1,\infty)$, there exists a constant $N=N(p, d, \|b\|_{C^0}, \|b\|_{W^\alpha_m}, \| \sigma\|_{C^2})$, such that for all $n \in \mathbb{N}$ we have 
		\begin{equs}\label{est.supXXn}
			\big\|\sup_{t\in[0,1]}|X_t-X^n_t|\big\|_{L_{p}(\Omega)}\leq N n^{-1/2}.
		\end{equs}
	\end{proposition}
	\begin{proof}
		The argument is similar to \cite[Section 7]{le2021taming}, however, the rate in \eqref{est.supXXn} is improved under \cref{asn:Sobolev}. The details are presented below.

		We raise \eqref{eqn.transformSDE} to $p$-th power to find that
		\begin{equs}
			\xi_t:=\sup_{s\in[0,t]}|X_s-X^n_s|^p
			\les |x_0-x_0^n|^p+\sum_{i=0}^6\sup_{s\in[0,t]}|J^i_s|^p.
		\end{equs}
		Using \eqref{tmp.2709regu} and H\"older's inequality, we get
		\begin{equs}
			\sup_{s\in[0,t]}(| J^0_s|^p+| J^1_s|^p) \les \theta^{-p/2}\sup_{s\in[0,t]}|X_s-X_s^n|^p+ \theta^{p/2}\left(\int_0^t \xi_r^{2/p} dr\right)^{\frac p2}.
		\end{equs}

		To estimate $J^4,J^5,J^6$, we will utilize a special case of the pathwise BDG inequality of \cite[Theorem 5]{Pietro}. Namely, there exists a constant $C=C(p,d)$ such that for any c\'adl\'ag martingale $\bar M$, there exists a local martingale $\hat M$ such that with probability one,
		\begin{equs}
			\sup_{s\in[0,t]}|\bar M_s|^p\le C[\bar M]_t^{p/2}+\hat M_t, \quad \forall t.
		\end{equs}
		In the above, $[\bar M]$ is the quadratic variation of $\bar M$.  
		Using the Lipschitz regularity of $\sigma$, \eqref{tmp.2709regu} and the pathwise BDG inequality, we can find a local martingale $M^4$ such that
		\begin{equs}
			\sup_{s\in[0,t]}| J^4_s|^p&\les \left(\int_0^t|X_r-X^n_r|^2dr\right)^{\frac p2}+M^4_t
			\\&\leq\left(\int_0^t \xi_r^{2/p}dA^n_r \right)^{\frac p2}+M^4_t.
		\end{equs}
		By \eqref{eq:Lip_Max},  for almost every $r\in[0,1]$ and almost every $x,y\in\R^d$, we have
		\begin{equs}
			|\nabla u(r,x)- \nabla u(r,y)|\les|x-y|(\cmm|\nabla^2u|(r,x)+\cmm|\nabla^2u|(r,y)).
		\end{equs}
		Applying this (keeping in mind that the laws of $X_r$ and $X^n_r$ are absolutely continuous) and the pathwise BDG inequality, we get
		\begin{equs}
			\sup_{s\in[0,t]}| J^6_s|^p
			&\les\left(\int_0^t|X_r-X^n_r|^2(\cmm|\nabla^2u|(r,X_r)+\cmm|\nabla^2u|(r,X^n_r))^2dr\right)^{\frac p2}+M^6_t
			\\&\leq \left(\int_0^t \xi_r^{2/p}dA^n_r \right)^{\frac p2}+M^6_t
		\end{equs}
		for some local martingale $M^6$.
		Using \cref{lem.Xtkt}, \eqref{tmp.2709regu}, Lipschitz regularity of $\sigma$ and the pathwise BDG inequality, we have
		\begin{equs}
		 	\sup_{s\in[0,t]}|J^5_s|^p\les\left(\int_0^t|X^n_r-X^n_{\kappa_n(r)}|^{2}dr\right)^{\frac p2}+M^5_t.
		\end{equs}
		It follows that
		\begin{equs}
			\xi_t\les \theta^{-1/2} \xi_t+ (1+\theta^{p/2})\left(\int_0^t \xi^{2/p}dA^n\right)^{\frac p2}+J^*_t+M_t,
		\end{equs}
		where 
  $$
  J^*_t=|x_0-x^n_0|^p+\sum_{i=2}^3\sup_{s\in[0,t]}|J^i_s|^p+\left(\int_0^t|X^n_r-X^n_{\kappa_n(r)}|^{2}dr\right)^{\frac p2}
  $$ 
  and $M=M^4+M^5+M^6$. 
		By choosing $\theta$ sufficiently large, this deduces to
		\begin{equs}
		 	\xi_t\les  \left(\int_0^t \xi^{2/p}dA^n\right)^{\frac p2}+J^*_t+M_t.
		\end{equs}
		Applying stochastic Gronwall lemma (\cite[Lemma 3.8]{le2021taming}), we have
		\begin{equs}\label{tmp.2709eA}
			\E e^{-c_p |A^n_1|^{\max(p/2,1)}} \xi_1\les \E J^*_1.
		\end{equs}
			We proceed with an estimate for the right-hand-side of the above inequality. Notice that by Sobolev embedding and \eqref{eq:W^2_p},  we have that $\nabla u \in L_{q}((0, 1); C^{1-\delta}(\R^d))$ for any $\delta \in (0, 1)$ and $q \geq 1$.   By applying \cref{lem:quad1} with the choice $g=I+\nabla u$ and  $f=b$, and \cref{lem:interchange-LpCa}, we have 
		\begin{equs}
			\E \sup_{s\in[0,1]}| J^2_s|^p
			\les n^{-p/2}.
		\end{equs}
		Using the Lipschitz continuity and the boundedeness of $\sigma$, by the Cauchy--Schwarz inequality, we have
		\begin{equs}
			\E \sup_{s\in[0,1]}| J^3_s|^p&\les
			\E\Big|\int_0^1|\nabla^2 u(r,X^n_r)||X^n_r-X^n_{\kappa_n(r)}| dr \Big|^p
			\\&\le \E\left(\int_0^t|\nabla^2 u(r,X^n_r)|^2dr\right)^{\frac p2}\left(\int_0^1|X^n_r-X^n_{\kappa_n(r)}|^2 dr \right)^{\frac p2}
			\\&\le \left[\E\left(\int_0^1|\nabla^2 u(r,X^n_r)|^2dr\right)^{p}\right]^{1/2} \left[\E\left(\int_0^1|X^n_r-X^n_{\kappa_n(r)}|^{2} dr \right)^{p}\right]^{1/2}.
		\end{equs}
		By Jensen's inequality, the Krylov-type estimate  from \cref{lem:Krylov_est_Xn} and \eqref{eq:W^2_p} we see that the first factor on the right-hand side above is bounded uniformly in $n$. Hence, applying \cref{lem.Xtkt}, we obtain that
		\begin{equs}
			\E \sup_{s\in[0,1]}|J^3_s|^p\les n^{-p/2}.
		\end{equs}
		The previous estimates for $\E \sup_{s\in[0,t]}|J^i_s|^p$, $i=2,3$ combined with \cref{lem.Xtkt} and \eqref{tmp.2709eA} imply that
		\begin{equs}
			\big\|e^{-c_p|A^n_1|^{\max(p/2,1)}}\sup_{t\in[0,1]}|X_t-X^n_t|\big\|_{L_p(\Omega)}\les n^{-1/2}.
		\end{equs}
		Under Assumptions \ref{asn:sigma} and \ref{asn:Sobolev}, it follows from  \cite[Lemma 7.4]{le2021taming} that $\sup_n\E e^{\kappa|A^n_1|^\rho}$ is finite for every $\kappa>0$ and $\rho>0$. In addition, by H\"older's inequality, 
		\begin{equs}
			\big\|\sup_{t\in[0,1]}|X_t-X^n_t|\big\|_{L_p(\Omega)}\le \big\|e^{c_{2p}|A^n_1|^p}\big\|_{L_{2p}(\Omega)}^{1/2} \big\|e^{-c_{2 p}|A^n_1|^p}\sup_{t\in[0,1]}|X_t-X^n_t|\big\|_{L_{2p}(\Omega)}^{1/2}  .
		\end{equs}
		Combining with the previous estimates, we obtain \eqref{est.supXXn}.
	\end{proof}
	\begin{proof}[Proof of \cref{prop.tight.sobolev}]
		Define $Z^n_t=X_t-X^n_t$. From \eqref{eqn.transformSDE}, we have
		\begin{equs}
			\|Z^n_{s,t}\|_{L_p(\Omega)}\le \sum_{i=0}^6 \| J^i_{s,t}\|_{L_p(\Omega)}.        \label{eq:Z_est_Js}
		\end{equs}
  Here and below we use the shorthand $f_{s,t}=f_t-f_s$.
  By the fundamental theorem of calculus we have 
  \begin{equs}
      J^0_{s, t} & = \int_0^1 \nabla u(t, X^n_t+ \theta (X_t-X^n_t))  Z^n_{s, t} \, d\theta
      \\ &\qquad+ \int_0^1 \big(\nabla u(t, X^n_t+ \theta (X_t-X^n_t))-\nabla u(s, X^n_s+ \theta (X_s-X^n_s)) \big)  Z^n_s \, d\theta.
  \end{equs}
 Let us fix  $\gamma \in (0, 1)$.  By using \eqref{tmp.2709regu} and \eqref{eq:Holder_grad_u} we get 
  \begin{equs}
     \|  J^0_{s, t} \|_{L_p(\Omega)} 
     \lesssim \theta^{-1/2}   \|   Z^n_{s, t} \|_{L_p(\Omega)} + \|(|t-s|^{\gamma/2} +| X^n_{s, t}|^{\gamma}+|X_{s, t}|^{\gamma}) |Z^n_{s}| \|_{L_p(\Omega)}
     \\
     \lesssim  \theta^{-1/2}   \|  Z^n_{s, t} \|_{L_p(\Omega)} + |t-s|^{\gamma/2}n^{-1/2},
  \end{equs}
  where for the last step we have used H\"older's inequality and  \cref{prop.weightedmoment}. 
		We estimate the remaining terms at the right hand side of \eqref{eq:Z_est_Js}  as in the proof of  \cref{prop.weightedmoment}, however this time, taking also into account \eqref{est.supXXn}. 
		From \eqref{tmp.2709regu} and \eqref{est.supXXn}, 
		\begin{equs}
			\| J^1_{s,t}\|_{L_p(\Omega)}\les\theta^{1/2}\int_s^t \|X_r-X^n_r\|_{L_p(\Omega)}dr\les \theta^{1/2} n^{-1/2}(t-s) .
		\end{equs}
		Next, notice that by Sobolev embedding and \eqref{eq:W^2_p},  we have that $\nabla u \in L_q((0, 1); C^{1-\delta}(\R^d))$ for any $\delta \in (0, 1)$ and $q\geq 1$. By applying \cref{lem:quad1} with the choice $g=I+\nabla u$ and $f=b$, we have that
		\begin{equs}   \label{eq:J_2-to-0}
			\| J^2_{s,t}\|_{L_p(\Omega)} \leq |t-s|^{1/2} n^{-1/2-\eps}. 
		\end{equs}
 By \eqref{eq:u_in_LqLp} it follows that $\nabla^2u$ satisfies the conditions of $g$ in \cref{lem:(ii)}, hence we get 
		\begin{equs}       \label{eq:J_3-to-0}   
			\|J^3_{s,t}\|_{L_p(\Omega)}\les n^{-1/2-\eps}(t-s)^{1/2}.
		\end{equs}
  By the BDG inequality, \eqref{eq:Holder_grad_u}, the Lipschitz continuity of $\sigma$ and \cref{prop.weightedmoment}, we get 
  \begin{equs}
			\| J^4_{s,t}\|_{L_p(\Omega)}\les  n^{-1/2} |t-s|^{1/2}.
		\end{equs}
  Similarly, by using \cref{lem.Xtkt} in place of \cref{prop.weightedmoment}, we get 
  \begin{equs}
			\| J^5_{s,t}\|_{L_p(\Omega)}\les \left\|\left(\int_s^t|X^n_r-X^n_{\kappa_n(r)}|^2dr\right)^{1/2}  \right\|_{L_p(\Omega)}\les n^{-1/2} |t-s|^{1/2}.
		\end{equs}
  By the BDG inequality, \eqref{est.supXXn}, H\"older's inequality, and \cref{prop.weightedmoment}, we get 
		\begin{equs}
			\|J^6_{s,t}\|_{L_p(\Omega)}\les \|\sup_{r\in[0,1]}|X_r-X^n_r| | A^n_{s,t}|^{1/2} \|_{L_p(\Omega)} \les n^{-1/2}  \|  A^n_{s,t} \|_{L_p(\Omega)}^{1/2}.
		\end{equs}
   Then, by H\"older's inequality, Krylov's estimate and its variant \cref{lem:Krylov_est_Xn}, we get 
  \begin{equs}
 \| A^n_{s,t} \|^p_{L_p(\Omega)} & \les |t-s|^{p-1} \E \int_s^t |\mathcal{M} |\nabla^2 u | (r, X_r) |^{2p} + |\mathcal{M} |\nabla^2 u |(r, X^n_r) |^{2p} \, dr +|t-s|^p
 \\
 & \les |t-s|^{p-1+1-(d+2)/(2q)} ( \| \nabla^2u \|^{2p}_{L_{2p q}((0, 1) \times \R^d)}+1),
 \end{equs}
provided that $q> (d+2)/2$. Hence, for $q$ large enough, the above combined with \eqref{eq:W^2_p} gives
 \begin{equs}
\| A^n_{s,t} \|^p_{L_p(\Omega)} & \les  |t-s|^{\gamma p},
   \end{equs}
   which implies that
   \begin{equs}
       \|J^6_{s,t}\|_{L_p(\Omega)}\les |t-s|^{\gamma/2} n^{-1/2}.
   \end{equs}
Combining all these estimates, gives, with some constant $N$ independent of $\theta$ and $n$,
\begin{equs}
   \|   Z^n_{s, t} \|_{L_p(\Omega)} \leq N(  \theta^{-1/2}   \|   Z^n_{s, t} \|_{L_p(\Omega)} + |t-s|^{\gamma/2}n^{-1/2} ).
\end{equs}
Hence, provided that  $\theta$ is sufficiently large, the result follows.		
	\end{proof}

 \begin{proof}[Proof of Theorems \ref{thm:existence_uniqueness_Sobolev} and \ref{thm:main_theorem_Sobolev}]
 Let us start with the uniqueness part of \cref{thm:existence_uniqueness_Sobolev}. Recall that the matrix $I- \nabla u (t, x)$ is invertible and both $I- \nabla u (t, x)$ and $(I- \nabla u (t, x))^{-1}$ are bounded uniformly in $(t, x)$. Then, it is easy to see that if $(X,V)$ satisfies \eqref{sde}-\eqref{eq:V_Sbolev_simplified}, then $Z:=\big(I-\nabla u(\cdot, X_\cdot)\big)V$ satisfies the linear equation
\begin{equs}   
		Z_t&=  \theta\int_0^t\nabla u(r,X_r) \big(I-\nabla u(t, X_t) \big)^{-1} Z_r\, dr
  \\
		&  \quad +\int_0^t[\nabla \big( (\nabla u + I ) \sigma \big) ](r,X_r)\big(I-\nabla u(t, X_t) \big)^{-1} Z_rdB_r
		\\&\quad+\frac1{\sqrt2}\int_0^t(I+\nabla u)\sigma \nabla \sigma  (r,X_r) dW_r.  
   \label{eq:Z}
	\end{equs}
For the system \eqref{sde},\eqref{eq:Z} we have that strong uniqueness holds, and by the result of Yamada and Watanabe \cite[Proposition 1]{Yamada_Watanabe} (see also \cite[Proposition 3.20, p. 309]{Karatzas}) uniqueness in law holds as well. Consequently, if  $\{ (\Omega, \cF,  \mathbb{F}, \bP),  X, V, B, W\}$ and $\{ (\Omega', \cF',  \mathbb{F}', \bP'),  X', V', B', W'\}$ are solutions of the system \eqref{sde},\eqref{eq:V}, then the laws of $(X, (I-\nabla u (\cdot, X_{\cdot})^{-1} V, B, W)$ and  $(X', (I-\nabla u (\cdot, X'_{\cdot})^{-1} V', B', W')$ on $\mathcal{C}$ coincide. Since the map $(x, y) \mapsto (x, (I-\nabla u(t, x))^{-1} y)  $ is a measurable bijection for each $t\geq 0$, it follows that the  finite dimensional distributions of $(X, V, B, W)$ and  $(X', V', B', W')$ coincide, hence, so does their distributions on $\mathcal{C}$. 

Next, we move to the proof of the existence part of \cref{thm:existence_uniqueness_Sobolev} and \cref{thm:main_theorem_Sobolev}. With the usual notation, let us set $S^n=(X, V^n, B, W^n)$. Clearly, because of the uniqueness part, it suffices to show that for any  subsequence $\{S^n\}_{n \in \mathbb{N}_1}  \subset  \{S^n\}_{n \in \mathbb{N}}$, there exists a further subsequence $\{S^n\}_{n \in \mathbb{N}_2}  \subset  \{S^n\}_{n \in \mathbb{N}_1}$ which converges weakly to a  solution of the system \eqref{sde},\eqref{eq:V_Sbolev_indices}. We fix  $\beta \in (0, 1/2)$, and by arguing as in the proof of Theorems \ref{thm:existence_uniqueness} and \ref{thm:main_Holder}, this time using the compactness result from Proposition \ref{prop.tight.sobolev}, we have that the exists a probability space 
     $( \hat \Omega, \hat \cF, \hat  \bP) $ and random variables 
	\begin{equs}
	(\hat{ \mathcal{X}}^{n},  \hat  X^n, \hat   V^n, \hat  B^n , \hat W^n)  : \hat  \Omega \to  \mathcal{C}^{\beta}=: \big(C^\beta([0, 1] ; \R^d)^4 \times \big) \times C^\beta([0, 1] ; \R^{d\times d})
	\end{equs}
	for $n \in \mathbb{N}_2$, 
	and 
	\begin{equs}
	(\hat{ \mathcal{X}}, \hat  X,  \hat   V, \hat  B, \hat W)   : \hat \Omega \to  \mathcal{C}^{\beta}, 
	\end{equs}
	such that 
	\begin{equs}      \label{eq:same-distribution-Sobolev}
	(\hat{ \mathcal{X}}^{n},  \hat  X^n, \hat   V^n, \hat  B^n , \hat W^n) \overset{d}{=}  (X, X^n, V^n, B, W^n) \ \text{ on  \   $ \mathcal{C}^{\beta}$}, 
	\end{equs}
	and  
	\begin{equs}               \label{eq:ghydf46_0}     
	(\hat{ \mathcal{X}}^{n},  \hat  X^n, \hat   V^n, \hat  B^n , \hat W^n) \to (\hat{ \mathcal{X}}, \hat  X,  \hat   V, \hat  B, \hat W)   \ \text{ as \ $n \to \infty$},
	\end{equs}
	$\hat \bP$-a.s. in $ \mathcal{C}^{\beta}$ and also in $L_p(\hat{\Omega}; \cC^\beta)$. By \eqref{eq:same-distribution-Sobolev}, \eqref{eq:ghydf46_0},  and Proposition \ref{prop.weightedmoment} it follows that $\hat{\X}=\hat{X} $. Theroefore, we have 
 \begin{equs}            \label{eq:almost-sure-convergence-Sobolev}
	(\hat{ \mathcal{X}}^{n},  \hat  X^n, \hat   V^n, \hat  B^n , \hat W^n) \to (\hat{X}, \hat  X,  \hat   V, \hat  B, \hat W)   \ \text{ as \ $n \to \infty$},
	\end{equs}
	$\hat \bP$-a.s. in $ \mathcal{C}^{\beta}$ and also in $L_p(\hat{\Omega}; \cC^\beta)$.
	Let $\{\hat \cF^n_t\}_{t \geq0}$ be the augmentation of the filtration generated by $\hat{ \mathcal{X}}^n, \hat B^n$, and $\hat W^n$.  It follows again that $\hat B^n$  is a $\{\hat \cF^n_t\}_{t \geq0}$-Brownian motion and that  the equalities \eqref{eq:idk2}, \eqref{eq:idk3},  and \eqref{eq:idk4}, are satisfied. Next, we show that  \eqref{eq:idk1} is also satisfied. Let us fix $t \in [0, 1]$ and let $\pi$ be a partition on $[0, t]$. Let $b_m$ be a mollification of $b$, and let us set 
 \begin{equs}
     \hat{\mathbb{S}}_{\pi}= x_0+ \sum_{[s, s'
     ] \in \pi} \sigma( \hat{\mathcal{X}}^n_s)  \hat{B}^n_{s, s'}, \ \ \hat{\mathbb{D}}_{\pi}=  \sum_{[s, s'
     ] \in \pi} b( \hat{\mathcal{X}}^n_s)  (s'-s), \ \  \hat{\mathbb{D}}^m_{\pi}=  \sum_{[s, s'
     ] \in \pi} b_m( \hat{\mathcal{X}}^n_s)  (s'-s), 
 \end{equs}
 and 
 \begin{equs}
     \hat{\mathbb{S}}= x_0+  \int_0^ t \sigma( \hat{\mathcal{X}}^n_s) \, d\hat{B}^n_s , \ \ \hat{\mathbb{D}}=  \int_0^t b( \hat{\mathcal{X}}^n_s)  \, ds , \ \  \hat{\mathbb{D}}^m=  \int_0^t b_m( \hat{\mathcal{X}}^n_s)  \, ds ,  
 \end{equs}
 We also  define $\mathbb{S}_{\pi}, \mathbb{D}_{\pi},\mathbb{D}^m _{\pi},\mathbb{S}, \mathbb{D}$ and, $\mathbb{D}^m$,  similarly but with $\hat{\mathcal{X}}^n$, $\hat{B}^n$ replaced by $X$, $B$. By 
 \eqref{eq:same-distribution-Sobolev} it follows that the distributions of $(\hat{\mathcal{X}}^n, \hat{B}^n)$ and $(X, B)$ coincide, which in particular implies that  $(\hat{\mathcal{X}}^n_t,\hat{\mathbb{S}}_{\pi}, \hat{\mathbb{D}}^m_{\pi})$ and $(X_t,\mathbb{S}_{\pi}, \mathbb{D}^m_{\pi})$ have the same distribution.  Using the continuity of $b_m$ and $\sigma$, we can let  let $|\pi| \to 0$, and conclude that  $(\hat{\mathcal{X}}^n_t,\hat{\mathbb{S}}, \hat{\mathbb{D}}^m)$ and $(X_t,\mathbb{S}, \mathbb{D}^m)$ have the same distribution. Next, we want to let $m \to \infty$, in order to conclude that $(\hat{\mathcal{X}}^n_t,\hat{\mathbb{S}}, \hat{\mathbb{D}})$ and $(X_t,\mathbb{S}, \mathbb{D})$ have the same distribution. For this, it suffices to show that $\mathbb{D}^m \to  \mathbb{D}$  and $\hat{\mathbb{D}}^m \to  \hat{\mathbb{D}}$  in probability. Notice that since $b_m \to b$ almost everywhere and  since the law of $X_s$ is absolutely continuous for any $s \in [0, 1]$, we can use the boundedness of $b$ in order to get from Lebesgue's theorem on  dominated convergence that $\mathbb{D}^m \to  \mathbb{D}$ in $L_1(\Omega)$, hence also in probability. Similarly we get that $\hat{\mathbb{D}}^m \to  \hat{\mathbb{D}}$  in probability. Consequently,  the laws of $(\hat{\mathcal{X}}^n_t,\hat{\mathbb{S}}, \hat{\mathbb{D}})$ and $(\mathcal{X}^n_t,\mathbb{S}, \mathbb{D})$ coincide. In particular, this means that $(\hat{\mathcal{X}}^n, \hat{B}^n)$ satisfies \eqref{eq:idk1}. 

 Next,  let $\hat{\mathbb{F}}:= \{\hat \cF_t\}_{t \in [0, 1]}$ be the augmentation of the filtration generated by $\hat{X}, V,  \hat B, \hat W$.   We claim that 
	$\{ ( \hat \Omega,  \hat  \cF,   \hat{\mathbb{F}}, \hat \bP), \hat{ X}, \hat V, \hat B, \hat W \}$ is a solution of the system  \eqref{sde},\eqref{eq:V_Sbolev_indices}.
 
 The fact that $ \hat B$ and $\hat W$ are  independent $\hat{\mathbb{F}}$-Brownian motions with values in $\R^d $ and $\mathbb{R}^{d\times d}$, respectively, is shown in the proof of Theorems \ref{thm:existence_uniqueness} and \ref{thm:main_Holder}. To show that  $\{(\hat{\Omega}, \hat{\cF}, \hat{\mathbb{F}}, \hat{\bP}), \hat{ \mathcal{X}}, \hat B\}$ is a solution of \eqref{sde}, notice that since the distributions of $(\hat{\mathcal{X}}^n, \hat{B}^n)$ and $(X,B)$ coincide for all $n\in \mathbb{N}$, we also get  that the distributions of $(\hat{\mathcal{X}}, \hat{B})$ and $(X,B)$ also coincide. From here,  we can repeat the above argument where we showed that $(\hat{\mathcal{X}}^n, \hat{B}^n)$ satisfies \eqref{eq:idk1} in order to obtain that $(\hat{\mathcal{X}}, \hat{B})$ solves \eqref{sde}. Moreover, since pathwise uniqueness holds for \eqref{sde}, it follows that $\mathcal{X}$ is adapted to the augmentation of the filtration generated by $\hat{B}$. In particular, this shows that $( \mathcal{X}, \hat B) \indep \hat{W}$, since $ \hat B \indep \hat{W}$ .
	Consequently, \eqref{item:def_sol_1_Sobolev} and \eqref{item:def_sol_2_Sobolev} from \cref{def:solution_of_system_Sobolev} are satisfied. Moreover, \eqref{item:def_sol_3_Sobolev} follows from the construction of $ \hat{\mathbb{F}}$ and the fact that with probability one $V \in C^\beta([0, 1])$.  
 
 Finally, we focus on showing \eqref{item:def_sol_4_Sobolev}. Let $u=(u^{(1)},\ldots,u^{(d)})$, where $u^{(\ell)}$, for $\ell \in \{1, ..., d\}$ is the solution to  \eqref{eqn.pde}.  Using It\^o's formula, we have
		\begin{equs}\label{eqn.zkinn}
			\hat{V}^n_t=\sqrt n\sum_{i=0}^6  \hat{J}^{i, n}_t,
		\end{equs}
		where the  $\hat{J}^{i, n}$ s are defined similarly to the $J^i$s in \eqref{eq:def_J} but  with $(\hat{\X}^n,\hat{X}^n,\hat{B}^n)$ in place of $(X,X^n,B)$.  We want to let $n \to \infty$ in the above equality and recover \eqref{eq:V_Sbolev_simplified}. For $\hat{J}^{0, n}_t$, we have 
  \begin{equs}
      \sqrt{n}\hat{J}^{0, n}_t = \hat{V}^n_t \int_0^1 \nabla u\big(t, \hat{X}^n_t+ \vartheta(\hat{\X}^n_t-\hat{X}^n_t) \big)  \, d \vartheta, 
  \end{equs} 
  and by using the continuity of $\nabla u$ and the convergence in \eqref{eq:almost-sure-convergence-Sobolev}, we get by virtue of Lebesgue's theorem on dominated convergence  that $\sqrt{n}\hat{J}^{0, n}_t \to \nabla u (t, \hat{X}_t) \hat{V}_t $ in probability. Similarly, we obtain that 
  \begin{equs}
      \sqrt{n}\hat{J}^{1, n}_t \xrightarrow{\hat{\bP}}  \theta\int_0^t\nabla u(r,\hat{X}_r)\hat{V}_rdr.
  \end{equs}
  As in \eqref{eq:J_2-to-0} and \eqref{eq:J_3-to-0}, we have that $\sqrt{n}\hat{J}^{2, n}_t \to 0$ and $\sqrt{n}\hat{J}^{3, n}_t \to 0$, in $L_p(\hat{\Omega})$, hence, also in probability. Then, we move to $\hat{J}^{4, n}$. We have that 
  \begin{equs}
      \sqrt{n}\hat{J}^{4, n}_t = \int_0^t \int_0^1 \nabla u (r, \hat{X}^n_r) \hat{V}^n_r \nabla \sigma \big(\hat{X}^n_r+\vartheta(\hat{\X}^n_r- \hat{X}^n_r) \big) \, d \vartheta \, dr.
  \end{equs}
  By \eqref{eq:almost-sure-convergence-Sobolev}, combined with the continuity of $\nabla \sigma$ and $\nabla u$ (see \eqref{eq:Holder_grad_u}),  by virtue of  Lebegue's dominated convergence theorem,  we have that 
  \begin{equs}
      \sqrt{n}\hat{J}^4_t \xrightarrow{\hat{\bP}} \int_0^t  \nabla u (r, \hat{X}_r) \hat{V}_r \nabla \sigma \big(\hat{X}_r) \big) \,  dr.
  \end{equs}
  Next, we see that 
  \begin{equs}
			\sqrt n \hat{J}^{5, n}_t
			&=
			\int_0^t\int_0^1 [I+ \nabla  u(r, \hat{X}^n_r)] \nabla \sigma(\hat{X}^n_r-\vartheta(\hat{X}^n_r-\hat{X}^n_{\kappa_n(r)}))d \vartheta b(\hat{X}^n_{\kappa_n(r)}) \sqrt{n} (r-\kappa_n(r)) \, d\hat{B}^n_r 
			\\&\quad+\frac1{\sqrt 2}\int_0^t\int_0^1 \nabla  u(r,\hat{X}^n_r) \nabla \sigma (\hat{X}^n_r-\vartheta(\hat{X}^n_r-\hat{X}^n_{\kappa_n(r)}))d \vartheta \sigma (\hat{X}^n_{\kappa_n(r)}) d\hat{W}^n_r.
		\end{equs}
		Since $\nabla u,\nabla \sigma$ and $b$ are bounded, the first term  vanishes in the limit. 
		For the second term, we can use the uniform continuity of $\nabla u,\sigma$, and  $\nabla \sigma$ together with \eqref{eq:almost-sure-convergence-Sobolev} to obtain, by virtue of Lemma \ref{lem:convergence_stochastic_integrals}, 
  \begin{equs}
      \sqrt n \hat{J}^{5, n}_t \xrightarrow{\hat{\bP}}  \frac1{\sqrt 2} \int_0^t  \nabla  u(r,\hat{X}_r) \nabla \sigma (\hat{X}_r)  \sigma (\hat{X}_r) d\hat{W}_r. 
  \end{equs}
  For the last term, we claim that 
  \begin{equs} \label{eq:hat_J^6_convergence}
    \sqrt n \hat{J}^{6, n}_t  \xrightarrow{\hat{\bP}}  \int_0^t   \nabla^2 u(r,\hat{X}_r)  \hat{V}_r \sigma(\hat{X}_r)  d\hat{B}_r. 
  \end{equs}
  To see this, let $p> (d+2)/2$ and let  $u_m$ be a smooth approximation of $u$ such that $\| u_m -u\|_{L_{4p}([0, 1]; W^2_p)} \to 0$, as $m \to \infty$.  For $\eps>0$, we have  
\begin{equs}
 \,   & \hat{\mathbf{P}} \left( \Big|  \sqrt n \hat{J}^{6, n}_t  - \int_0^t   \nabla^2 u(r,\hat{X}_r)  \hat{V}_r \sigma(\hat{X}_r)  d\hat{B}_r \Big| > 3 \eps \right) 
   \\
     \leq  \, 
  & \hat{\mathbf{P}} \left( \Big| \int_0^t   \sqrt n \big[ (\nabla u - \nabla u_m)(r,  \hat{\mathcal{X}}^n_r)-  (\nabla u - \nabla u_m)(r,  \hat{X}^n_r) \big] \sigma(\hat{\mathcal{X}}^n_r) \,   d\hat{B}^n_r \Big| > \eps \right)  
  \\
  & \, +\hat{\mathbf{P}} \left( \Big| \int_0^t   \sqrt n \big[ \nabla u_m(r,  \hat{\mathcal{X}}^n_r)-   \nabla u_m(r,  \hat{X}^n_r) \big] \sigma(\hat{\mathcal{X}}^n_r) \,   d\hat{B}^n_r  - \int_0^t   \nabla^2 u_m (r,\hat{X}_r)  \hat{V}_r \sigma(\hat{X}_r)  d\hat{B}_r\Big| > \eps \right) 
  \\
  & \, +\hat{\mathbf{P}} \left( \Big| \int_0^t  \big[ \nabla^2  u_m(r,  \hat{X}_r)-   \nabla^2 u(r, \hat{X}_r) \big]  \hat{V}_r \sigma(\hat{X}_r) \,  d\hat{B}_r\Big| > \eps \right)
  \\
 \,   =& \, A_1^{n, m}+A_2^{n,m}+A_3^m . 
  \end{equs}
For $A_1$, we use Markov's inequality, It\^o's isometry, the boundedness of $\sigma$, and \eqref{eq:Lip_Max} to obtain 
\begin{equs}
    A^{n, m}_1 \leq \eps^{-2} \| \sigma \|_{C^0}\E \int_0^t |V^n_r|^2   \Big( \mathcal{M}| \nabla^2 (u_m- u) | (r,\hat{\mathcal{X}}^n_r)-  \mathcal{M}| \nabla^2 (u_m- u) | (r,\hat{X}^n_r)   \Big)^2 \, dr.
\end{equs}
Next, we use the the Cauchy-Schwartz inequality, the uniform bound on $\hat{V}^n$  which follows from Proposition \ref{prop.tight.sobolev} and \eqref{eq:same-distribution-Sobolev} and Krylov's estimate as well as its discrete version from Lemma \ref{lem:Krylov_est_Xn}, and \eqref{eq:Hardy-Littlewood},  to see that 
\begin{equs}
    \sup_n A^{n, m}_1 &\leq    N \eps^{-2} \Big(\E \int_0^t\big( \mathcal{M}| \nabla^2 (u_m- u) | (r,\hat{\mathcal{X}}^n_r)-  \mathcal{M}| \nabla^2 (u_m- u) | (r,\hat{X}^n_r)   \big)^4 \, dr \Big)^{1/2}
    \\
    &\leq   N \eps^{-2} \| \mathcal{M}| \nabla^2 (u_m- u) |^ 4 \|_{L_p([0, 1] \times \R^d)}^{1/2} 
    \\
    &\leq   N \eps^{-2} \|  \nabla^2 (u_m- u)  \|_{L_{4p}([0, 1] \times \R^d)}^{2p}  \to 0,  
\end{equs}
as $m \to \infty$. In a similar manner, we get that $A^m_3 \to 0$
as $m \to \infty$. For $A^{n,m }_2$, we claim that for fixed $m \in \mathbb{N}$, $A^{n,m}_2 \to 0 $, as $n \to \infty$. Indeed, for fixed $m$, since $u_m$ and $\sigma$ are sufficiently regular, it is easy to see by \eqref{eq:almost-sure-convergence-Sobolev} and the fundamental theorem of calculus that 
\begin{equs}
   \sup_{ r \in [0, 1]}| \sqrt n \big[ \nabla u_m(r,  \hat{\mathcal{X}}^n_r)-   \nabla u_m(r,  \hat{X}^n_r) \big] \sigma(\hat{\mathcal{X}}^n_r) \, -  \nabla^2 u_m (r,\hat{X}_r)  \hat{V}_r \sigma(\hat{X}_r)|  \xrightarrow{\hat{\bP}} 0. 
\end{equs}
By the above combined with \eqref{eq:almost-sure-convergence-Sobolev} and \cref{lem:convergence_stochastic_integrals}, the claim follows. 
Consequently, we have shown \eqref{eq:hat_J^6_convergence}. 

Finally, we can now go back to \eqref{eqn.zkinn} and let $n \to \infty$ in order to see that the processes $(\hat{X}, \hat{V}, \hat{B}, \hat{W})$ indeed satisfy \eqref{eq:V_Sbolev_simplified}, that is, \eqref{item:def_sol_4_Sobolev} from \cref{def:solution_of_system_Sobolev} is satisfied. This finishes the proof.

 \end{proof}

\begin{appendix}
    
\section{}

The following is a straightforward consequence of \cite[Theorem 2.2]{KurtzProtter}.
\begin{lemma}                \label{lem:convergence_stochastic_integrals}
    For $n \in \mathbb{N}\cup\{0\}$, let  $B^n$ be continuous $\mathbb{F}^n$-martingales and let $G^n$ be c\'adl\'ag  $\mathbb{F}^n$-adapted processes. Assume that $G^n \to G^0$ and $B^n \to B^0$, uniformly on $[0, 1]$, in probability. Then,   in probability, we have
    \begin{equs}
    \sup_{t \in [0, 1]}\big|\int_0^t G^n_s \, dB^n_s - \int_0^t G^0_s \, dB^0_s\big | \to 0.    
    \end{equs}
\end{lemma}

	\begin{lemma}          \label{lem:interchange-LpCa}
Let $Y$ be a process on $[0, 1]$ and $\beta \in (0,1)$. For any $\eps \in (0, \beta)$, there exists $p_0 \geq 0$ such that for all $p \geq p_0$ we have 
\begin{equs} \nonumber
[ Y ]_{L_p(\Omega; C^{\beta-\eps}([0, 1]))} \leq N [ Y]_{C^{\beta}([0, 1];L_p(\Omega))}, 
\\ \nonumber
\| Y\|_{L_p(\Omega; C^{\beta-\eps}([0, 1]))} \leq N \| Y\|_{C^{\beta}([0, 1];L_p(\Omega))}, 
\end{equs}
where $N$ depends only on $p, \beta$ and $\eps$. 
\end{lemma}

The proof of the above lemma is a simple application of Kolmogorov's continuity criterion or of the Sobolev Slobodeckij embedding theorem. 

\begin{lemma}               \label{lem:interpolation_PDE}
Under Assumptions \ref{asn:sigma} and \ref{asn:Sobolev},  estimate \eqref{eq:u_in_LqLp} holds. 
\end{lemma}

\begin{proof}
Throughout the proof, as usual,  the constants $N$ might change from line to line but only depends on $\lambda, \|b\|_{C^0(\R^d)}, \|b\|_{L_m(\R^d)}, d, \theta, \|\sigma\|_{C^2(\R^d)}$, and $\alpha$.  Let $u^{(\ell)}$ be the solution of \eqref{eqn.pde}. By \cite[Theorem 2.1]{Krylov_Mixed} we have that 
\begin{equs}
    \| u\|_{L_{p_0}((0, 1); W^2_m(\R^d))} \leq N \|b^{(\ell)}\|_{L_{p_0}((0, 1); L_m(\R^d)) } \leq N \| b^{(\ell)}\|_{W^\alpha_m(\R^d) } < \infty.  \label{eq:finite_W2m}
\end{equs}
Next, let us consider the auxiliary equation 
  \begin{equs}\label{eqn.pde_aux}
	\partial_t v+\frac{1}{2}a^{(i,j)}\partial^2_{i,j} v=f, \qquad u(1, \cdot)= 0, 
	\end{equs}
 for $f \in C^\infty_c((0, 1) \times \R^d )$. It is known that it admits a unique classical solution which in addition is three times continuously differentiable in space. By  \cite[Theorem 2.1]{Krylov_Mixed} again, we have 
 \begin{equs}
     \| v\|_{L_{p_0}((0, T);W^2_m(\R^d))} \leq N \|f\|_{L_{p_0}((0, T);L_m(\R^d))}.
 \end{equs}
 Moreover, by differentiating the equation in the spatial variable, in a similar manner we obtain that 
  \begin{equs}
     \| v\|_{L_{p_0}((0, T);W^3_m(\R^d))} \leq N \|f\|_{L_{p_0}((0, T);W^1_m(\R^d))}. 
 \end{equs}
 The above two estimates, by means of interpolation, imply that 
   \begin{equs}
     \| v\|_{L_{p_0}((0, T);W^{2+\alpha}_m(\R^d))} \leq N \|f\|_{L_{p_0}((0, T);W^\alpha_m(\R^d))}. \label{eq:interpolated_pde_estimate}
 \end{equs}
Next, from \eqref{eq:finite_W2m}, it follows that $u^{(\ell)}, b^{(\ell)}, b\cdot\nabla u^{(\ell)} \in L_{p_0}((0, T);W^{2+\alpha}_m(\R^d))$. Hence, for $n \in \mathbb{N}$,  there exists $g^n \in C^\infty_c((0, 1) \times \R^d)$ such that $g^n \to g:= \theta u^{(\ell)}-b^{(\ell)}-b\cdot\nabla u^{(\ell)}$ in $L_{p_0}((0, T);W^\alpha_m(\R^d))$, an $n \to \infty$, and $\|g_n\|_{L_{p_0}((0, T);W^\alpha_m(\R^d))} \leq \|g\|_{L_{p_0}((0, T);W^\alpha_m(\R^d))}$, for all $n \in \mathbb{N}$. 
Moreover, let $u_n$ be the solution of \eqref{eqn.pde_aux} with $f$ replace by $g^n$. 
By \eqref{eq:interpolated_pde_estimate}, we have 
\begin{equs}
     \| u_n\|_{L_{p_0}((0, T);W^{2+\alpha}_m(\R^d))} & \leq N \|g_n\|_{L_{p_0}((0, T);W^\alpha_m(\R^d))}
     \\
    & \leq N \|\theta u^{(\ell)}-b^{(\ell)}-b\cdot\nabla u^{(\ell)} \|_{L_{p_0}((0, T);W^\alpha_m(\R^d))} 
    \\
    & \leq  N \| b^{(\ell)}\|_{W^\alpha_m(\R^d) }, 
\end{equs}
where for the last inequality we have used \eqref{eq:finite_W2m}, Assumption \ref{asn:Sobolev}, and the boundedness of $\nabla u^{(\ell)}$ from \eqref{tmp.2709regu}. Moreover, it is known by standard solvability theory that $(u_n,\nabla u_n, \nabla^2 u_n) \to (u, \nabla u, \nabla^2 u) $ almost everywhere in $(0, 1) \times \R^d$, as $n \to \infty$. Hence, the desired estimate follows from Fatou's lemma. 
\end{proof}
\end{appendix}

\bibliographystyle{Martin}
\bibliography{Bounded_Drift}

\end{document}